\documentclass[11pt,reqno]{amsart}
\usepackage[utf8]{inputenc}
\usepackage[margin=2.4cm]{geometry}
\usepackage{amsmath}
\usepackage{amssymb}
\usepackage{enumerate}
\usepackage[encapsulated]{CJK}
\usepackage{amsthm}
\usepackage{graphicx}
\usepackage{comment}
\usepackage{xcolor,paralist,hyperref,fancyhdr,etoolbox}
\usepackage{subfig}

\theoremstyle{definition}
\newtheorem{definition}{Definition}[section]
\newtheorem{theorem}{Theorem}[section]
\theoremstyle{bfremark}
\newtheorem{remark}{Remark}[section]

\newtheorem{proposition}{Proposition}[section]
\newtheorem{lemma}{Lemma}[section]
\numberwithin{equation}{section}

\newcommand{\veph}{\hat\varepsilon }

\newcommand{\vepe}{\varepsilon^{\textup{ex}}}

\usepackage{mathtools}
\newcommand{\mchi}{\raisebox{0pt}[1ex][1ex]{$\chi$}}

\usepackage{lipsum}

\begin{document}
\title{Nonradial Quenching Profile for a MEMS Model} 
\author{Hsuan-Lin Liao and Van Tien Nguyen}
\date{\today}
\address{Hsuan-Lin Liao, Department of Mathematics, National Taiwan University.}
\email{d14221002@ntu.edu.tw}

\address{Van Tien Nguyen, Department of Mathematics, Institute of Applied Mathematical Sciences, National Taiwan University.}
\email{vtnguyen@ntu.edu.tw}
\maketitle

\let\thefootnote\relax

\begin{abstract}
We construct a quenching solution to the parabolic MEMS model
\[
u_t = \Delta u - \frac{1}{u^2} \quad \text{in } \mathcal{B} \times (0,T), \quad u|_{\partial \mathcal{B}} = 1,
\]
where \(\mathcal{B}\) is the unit disc in \(\mathbb{R}^2\), and \(T > 0\) denotes the quenching time.  The constructed solution quenches only at the origin and admits the final profile
\[
u(x,T) \sim \left(x_1^2 x_2^2 + \theta(x_1^6 + x_2^6)\right)^{\frac{1}{3}} \quad \text{as } |x| \to 0,
\]
where $\theta \in (0, \theta^*)$ for some $\theta^* > 0$. To our knowledge, this is the first example of a quenching solution with a genuinely non-radial profile. The proof relies on the construction of a good approximate solution, using a perturbative expansion in self-similar variables. We then justify the true solution that remains close to this approximation through a spectral analysis combined with a robust energy method.

\end{abstract}

\section{Introduction}
\label{section1}
\subsection{The MEMS Model}
Micro-electro-mechanical systems (MEMS) are miniature devices that integrate mechanical and electrical components on a micrometer scale. Typical applications include devices such as pressure sensors, accelerometers, microactuators, RF switches, and optical mirrors. MEMS devices often consist of a rigid ground plate and an elastic membrane suspended above it at an initial distance \( d > 0 \). The distance between them over time, referred to as the \emph{deflection}, can be modeled by partial differential equations under suitable assumptions. When the membrane is thin and the initial distance is small compared to the overall size of the device, one arrives at the following MEMS model (see \cite{kavallaris2018non} for a derivation):
\begin{align}
\begin{cases}
    \beta u_{tt} + u_t - \Delta u + \delta \Delta^2 u = \dfrac{-\lambda}{u^2 \left(1 + \alpha \int_{\Omega} \frac{1}{u} \, dx \right)^2} \quad \text{in} \quad \Omega \times (0,T), \\
    u = 1 \quad \text{on} \quad \partial\Omega \times (0,T), \\
    0<u_0(x) \leq 1,\label{omems}
\end{cases}
\end{align}
Here, \( u(x,t) \) denotes the deflection of the membrane. The parameter \(\beta>0\) measures the thickness of the membrane, \( \delta > 0 \) quantifies the relative importance of tension and rigidity, \( \lambda > 0 \) is proportional to the applied voltage, and \( \Omega \subset \mathbb{R}^n \) is a bounded domain. The boundary condition \( u = 1 \text{ on } \partial \Omega \times (0, T)\) reflects a fixed edge for the membrane. For more physical background and possible applications, we refer the reader to \cite{guo2005touchdown}, \cite{esposito2010mathematical}, \cite{flores2007analysis}, \cite{guo2014recent}, and \cite{pelesko2002modeling}.   \\

The model \ref{omems} is rich from a mathematical perspective. Taking limits of the parameters, \(\beta, \delta\) completely changes the type of the equation: from a second-order parabolic (\(\beta, \delta \ll1\)) to a
fourth-order parabolic (\(\beta \ll1, \delta \gg 1\)), or from parabolic (\(\beta \ll 1\)) to hyperbolic \((\beta \gg 1)\). In this work, we consider the case where \( \Omega = \mathcal{B} \), the open unit disc in \( \mathbb{R}^2 \), and set the parameters, \(\lambda = 1\), \(\alpha = 0\), while letting \(\beta \to 0\) and \(\delta \to 0\), for which the original model reads as
\begin{align}
\begin{cases}
    u_t = \Delta u - \dfrac{1}{u^2} & \text{in } \mathcal{B} \times (0, T), \\
    u = 1 & \text{on } \partial\mathcal{B} \times (0, T), \\
    0 < u_0(x) \leq 1. & \label{mems}
\end{cases}
\end{align}

According to \cite{kavallaris2018non}, a function \(u\) is said to be a classical solution of (\ref{mems}) if \( u \in C^{2,1}(\Omega \times (0,T)) \) and \(0<u \leq 1\). Moreover, the local Cauchy problem for (1.2) is solved, and the solution either exists globally or quenches in finite time \( T > 0 \), meaning
\[
\liminf_{t \to T} \left( \min_{x \in \overline{\mathcal{B}}} u(x,t) \right) = 0.
\]
Quenching corresponds physically to the moment when the elastic membrane contacts the ground plate, often referred to as \emph{touchdown} phenomena. Once this occurs, the MEMS device typically fails or becomes non-functional.

Quenching solutions of the MEMS model have received significant attention over the past few decades, particularly in the parabolic case, obtained by taking \(\beta \to 0\) and \(\delta \to 0\) in \eqref{omems}. For the local problem (\(\alpha=0\)), Deng and Levine \cite{deng1989blow} proved that the set of quenching points is contained within a compact subset of \(\Omega\), provided that \(\Omega\) is convex. Moreover, in one spatial dimension, the quenching set is finite under suitable assumptions; see \cite{GUO1992507} for further details. For the one-dimensional nonlocal problem (\(\alpha=1\)), Guo, Hu, and Wang \cite{guo2009nonlocal} showed that for specific initial data and sufficiently large \(\lambda\), the solution quenches only at the origin and forms a cusp at the quenching point without a precise quenching rate. In higher dimensions, Guo and Kavallaris \cite{guo2012nonlocal} demonstrated that quenching occurs provided the initial energy is sufficiently small. Explicit profiles for quenching solutions have been relatively scarce in the literature, apart from the works of Merle and Zaag \cite{merle1997reconnection}, Filippas and Guo \cite{filippas1993quenching}, Duong and Zaag \cite{duong2018construction}. In this work, we aim to construct a truly non-radial quenching solution to \eqref{mems} with a precise description of its behavior near the singularity.

\subsection{Main Theorem}
The following is a brief version of the main theorem. 

\begin{theorem}[Rough statement of the main theorem] There is smooth initial data $0 < u_0 \leq 1$ such that the corresponding solution $u(x,t)$ to \eqref{mems} quenches in finite time $T > 0$ only at the origin and admits the quenching rate 
\begin{equation}
   u(0,t ) = \min_{x \in \bar{\mathcal{B}}} u(x,t) \sim (T-t)^{\frac{1}{3}}, \quad \textup{as}\;\; t \to T,
\end{equation}
and the nonradial final profile 
\begin{equation}
    u(x,T) \sim \left(x_1^2 x_2^2 + \theta(x_1^6 + x_2^6)\right)^{\frac{1}{3}},\quad \textup{as} \;\; |x| \to 0,
\end{equation}
where $\theta \in (0, \theta^*)$ for some $\theta^* > 0$, chosen such that the estimates in Lemma \ref{lem:initial data bound} hold.
\end{theorem}  
To better restate the main theorem in detail, we introduce the self-similar variables,
\begin{align}
     u(x,t) = (T-t)^{\frac{1}{3}}\; w(y,s), \quad y = \frac{x}{\sqrt{T - t}}, \quad s = -\log(T - t),\label{selfsimvari}
\end{align}
which rescales the solution near the quenching time \(T\). In these variables, we define the quenching profile \(\mathcal{P}_{\theta}\) (see section \ref{section2} for a detailed derivation) by
\begin{align}
        \mathcal{P}_{\theta}(y,s) = \Psi_{\theta}(z,s)\mchi_{\frac{1}{4}\kappa(s)}(z) + \mathcal{C}(y,s) \mchi_K(z)+e^{\frac{s}{3}}\left(1-\mchi_{\frac{1}{4}\kappa(s)}(z)\right), \quad z=ye^{-\frac{s}{4}},\label{qp}
\end{align}
where \(\theta \in (0, \theta^*)\) for some $\theta^* > 0$, $K > 0$ a large fixed constant, \(\kappa(s)=e^{\frac{s}{4}}\), \(\mchi_K\) and \(\mchi_{\kappa(s)}\) are defined as in \ref{def:chiR} and $\Psi_{\theta}$ is the main non-radial profile 
\begin{align*}
        \Psi_{\theta}(z,s) = \left(3 + z_1^2 z_2^2 + \theta e^{-\frac{s}{2}}(z_1^6 + z_2^6)\right)^{\frac{1}{3}}.
\end{align*}
Moreover, the correction $\mathcal{C}(y,s)$ takes the form
\begin{align}
\begin{split}
    \mathcal{C}(y,s) =\ & \frac{\hbar}{9} e^{-s} (-2 y_1^2 - 2 y_2^2 + 4) \\
    & + e^{-2s} \Bigg[ 
        \frac{64}{81\hbar} h_0 h_0 
        + \frac{32}{27\hbar} (h_0 h_2 + h_2 h_0) 
        + \frac{64}{27\hbar} h_2 h_2 \\
    & \quad  -\frac{8}{27\hbar}s\, (h_2 h_4 + h_4 h_2) 
        + \frac{8}{27\hbar} (h_0 h_4 + h_4 h_0) \\
    & \quad + \frac{\hbar}{9} \delta (-30 y_1^4 + 180 y_1^2 - 120 - 30 y_2^4 + 180 y_2^2 - 120) \\
    & \quad - \frac{2}{54\hbar} \left( -12 y_1^4 y_2^2 + 12 y_1^4 - 12 y_1^2 y_2^4 + 144 y_1^2 y_2^2 
        - 144 y_1^2 + 12 y_2^4 - 144 y_2^2 + 144 \right)
    \Bigg],\label{corr}
\end{split}
\end{align}
where \(\hbar=3^{\frac{1}{3}}\) and \(h_m(\xi)\) denotes the rescaled Hermite polynomials,
\begin{align}
    h_m(\xi) = \sum_{n = 0}^{\lfloor m/2 \rfloor} \frac{m!}{n!(m - 2n)!} (-1)^n \xi^{m - 2n}.\label{hij}
\end{align}
The following is the main result of this paper. 
\begin{theorem}[Existence of quenching solutions to \eqref{mems} with a detailed asymptotic description]
\label{Theorem 1.1}
There is smooth initial data \(u_0\) with 
\(0<u_0 \leq 1\), such that the corresponding solution of (\ref{mems}) quenches only at the origin in finite time \(T>0\), with the nonradial quenching profile \(\mathcal{P}_{\theta}\). Moreover, the solution \(u\) admits the following asymptotic behavior:\\
(i) \textup{(Inner expansion)}\\
\begin{align}
   w(y,s)-\left(\left(3+e^{-s}y_1^2y_2^2+\theta e^{-2s}(y_1^6+y_2^6)\right)^{\frac{1}{3}}+\mathcal{C}(y,s)\right)=O(e^{-\frac{32}{12}s}),\label{innex}
\end{align}
uniformly in compact sets $\{|y| \leq R\}$ for $R > 0$ arbitrary, where $\theta \in (0, \theta^*)$ for some $\theta^* > 0$.\\
(ii) \textup{(Intermediate expansion)}\\
\begin{align}
    &\sup_{K \leq |y| \leq e^{\frac{s}{4}}}\left|\frac{w(y,s)-\mathcal{P}_{\theta}(y,s)}{|y|^{\frac{\alpha}{2}}}\right|=O(e^{-\frac{32}{12}s}), \quad \text{where } \alpha=18,\label{intex}\\
    &\sup_{e^{\frac{s}{4}} \leq |y| \leq e^{\frac{69}{196}s}}\left|\frac{w(y,s)-\mathcal{P}_{\theta}(y,s)}{|y|^{\frac{\gamma}{2}}}\right|=O(e^{-(\frac{5}{12}+\frac{\gamma}{8})s}), \quad \text{where } \gamma=8.01
\end{align}
and \(K \gg 1\) is a large fixed constant.\\
(iii) \textup{(Outer expansion)}\\
\begin{align}
    \sup_{e^{\frac{69}{196}s} \leq |y| \leq e^{\frac s 2}} \left|\frac{w(y,s)-\mathcal{P}_{\theta}(y,s)}{\mathcal{P}_{\theta}(y,s)}\right|\ll1.\label{outex}
\end{align}\\
(iv) \textup{(Final profile)} There is \(u^* \in C(\mathbb{R}^2 \setminus \{0\})\) such that \(u(x,t) \to u^{*}(x)\) uniformly on compact subsets of \(\mathbb{R}^2 \setminus \{0\}\) as \(t \to T\). Moreover,
\begin{align}
     u^{*}(x)\sim\left(x_1^2x_2^2+ \theta (x_1^6+x_2^6)\right)^{\frac{1}{3}} \quad \text{ as } \quad |x|  \to  0. \label{finpro}
\end{align}
\end{theorem}

\begin{remark}
We note that the profiles in \eqref{innex} and \eqref{finpro} are cross-shaped. To illustrate this, we display below the contour lines of the leading profile in the blowup variable
\[
\left(3 + z_1^2z_2^2 + \theta e^{-\frac{s}{2}}(z_1^6 + z_2^6)\right)^{\frac{1}{3}}, \quad z=ye^{-\frac{s}{4}},
\]
for \(\theta=\frac{1}{2}\) and \(s = 10\), \(100\), and \(1000\).

\vspace{10pt}

\begin{center}
\begin{minipage}{0.3\linewidth}
\includegraphics[width=\linewidth]{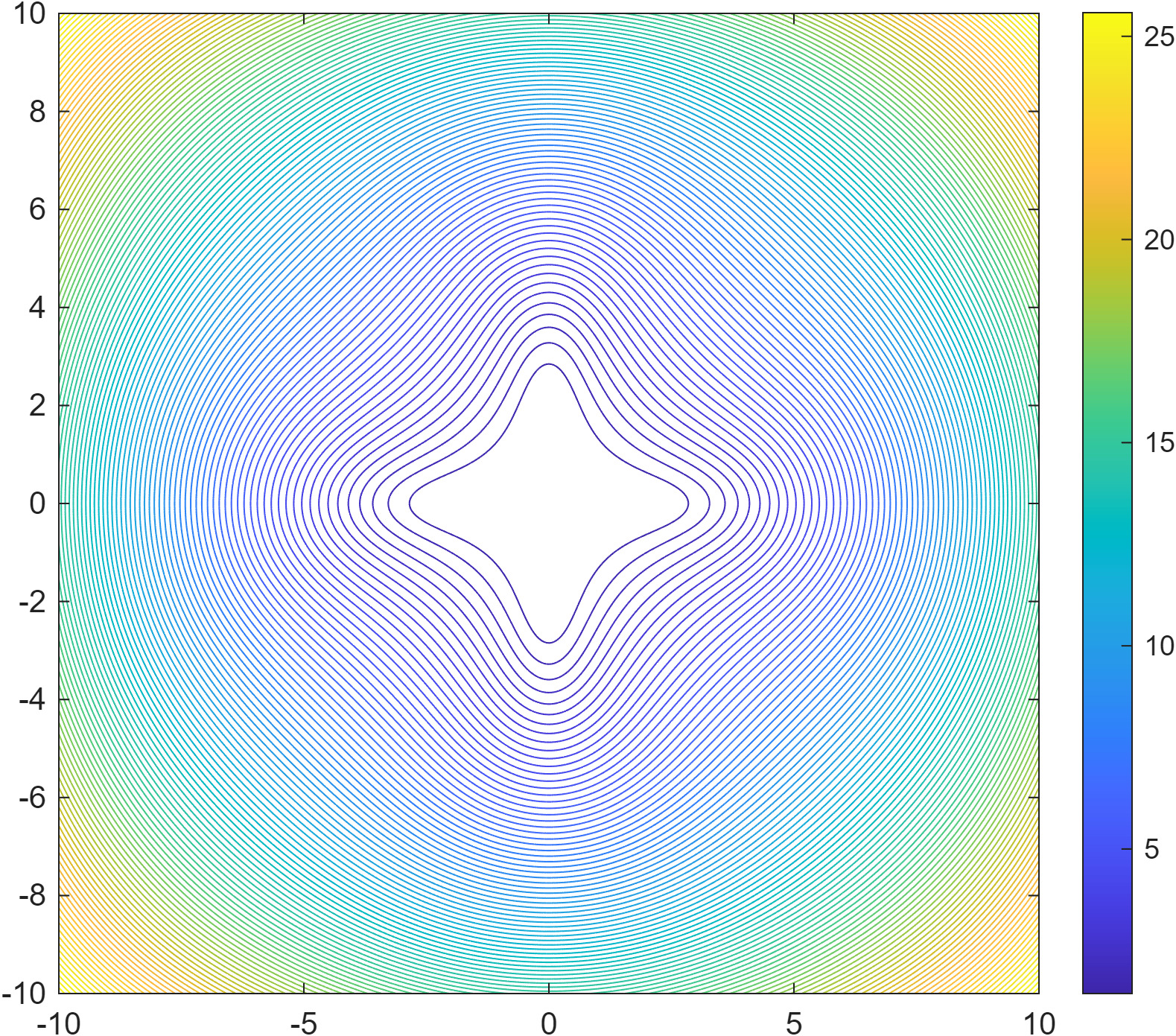}\\
\centering $s = 10$
\end{minipage}
\hfill
\begin{minipage}{0.3\linewidth}
\includegraphics[width=\linewidth]{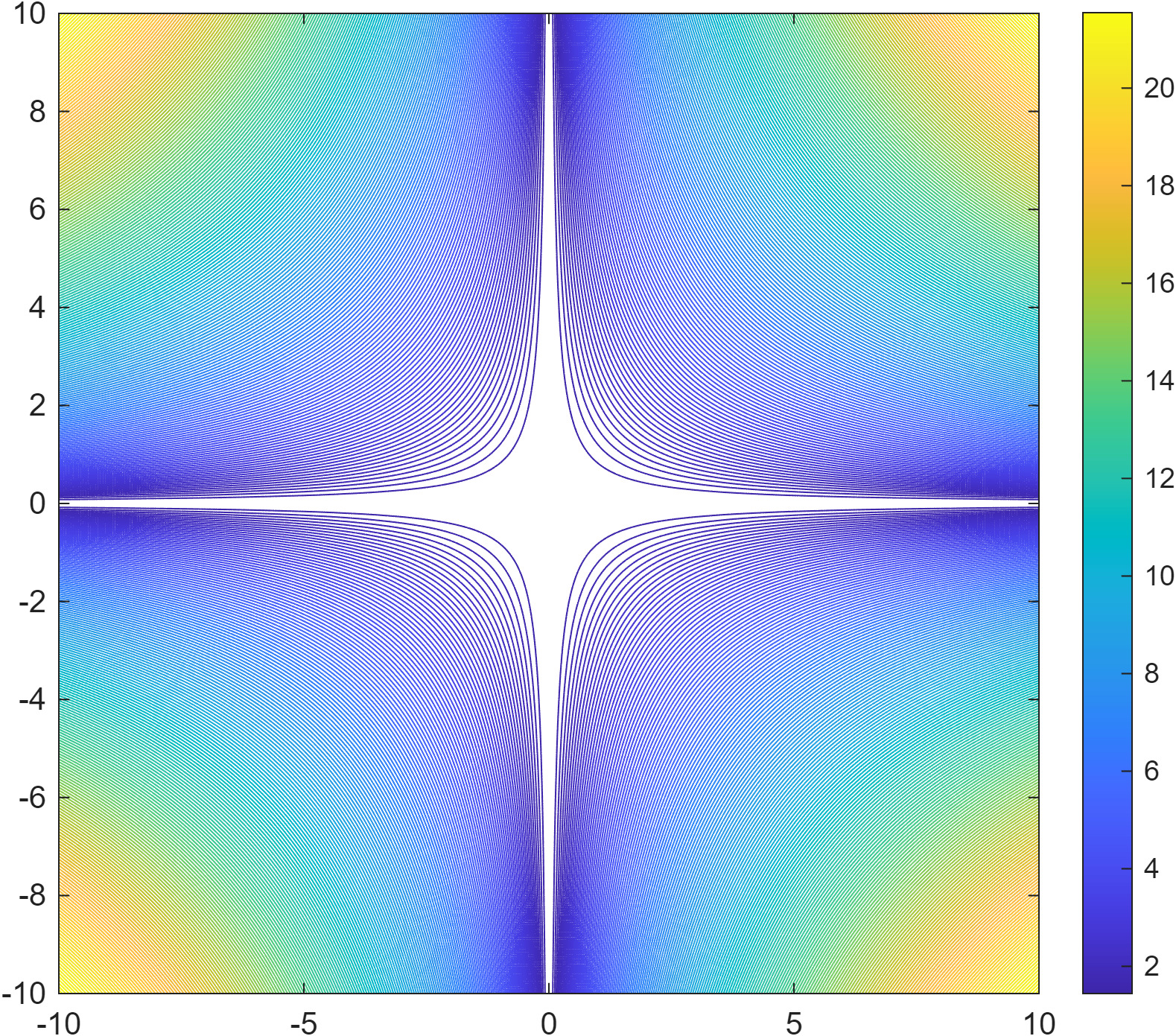}\\
\centering $s = 100$
\end{minipage}
\hfill
\begin{minipage}{0.3\linewidth}
\includegraphics[width=\linewidth]{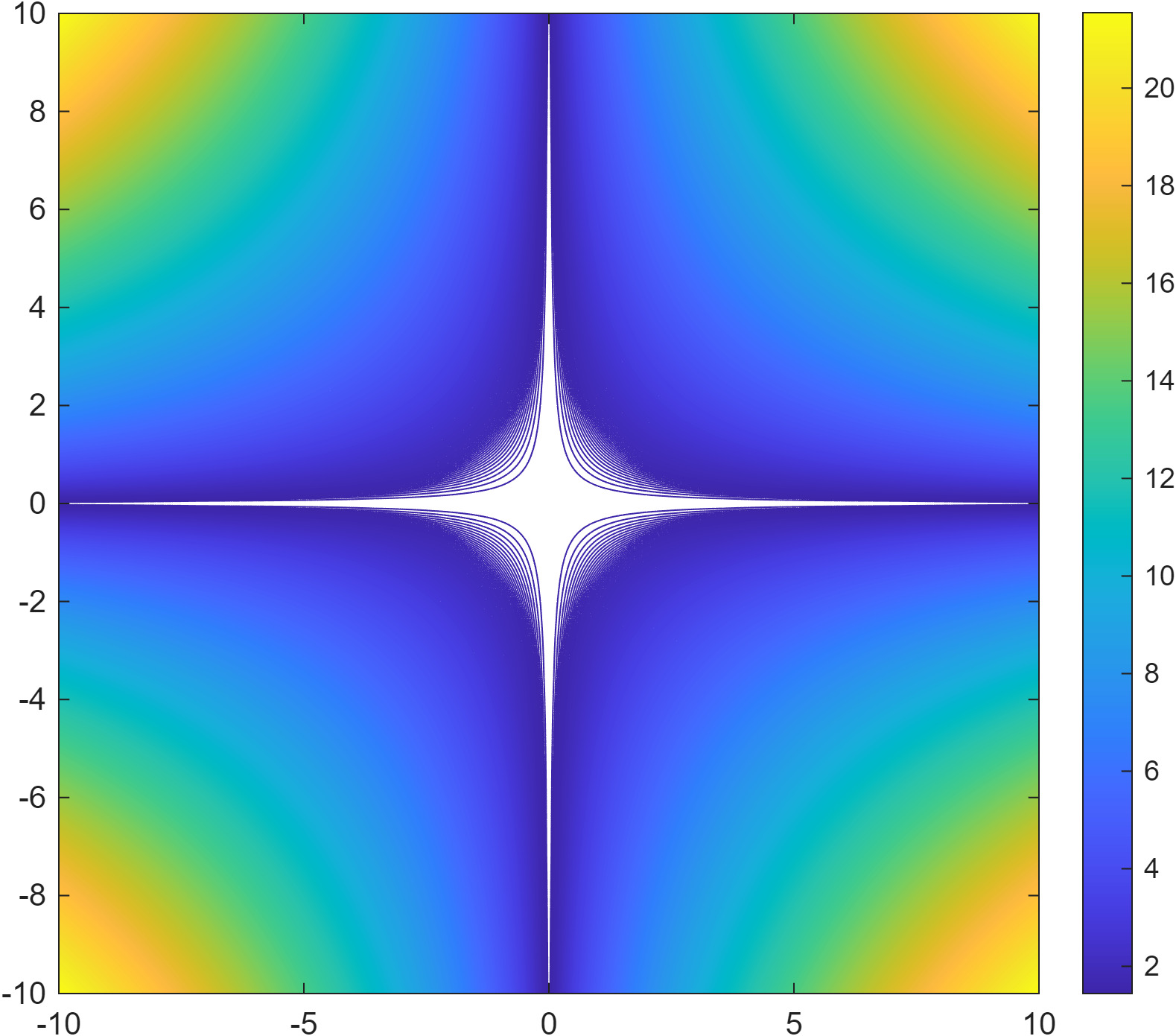}\\
\centering $s = 1000$
\end{minipage}
\end{center}
\vspace{10pt}
We can see that the cross-shaped profile becomes increasingly concentrated near the origin as \(s \to \infty\).
\end{remark}

\begin{remark}
    The correction term \(\mathcal{C}(y,s)\) in the definition of \(\mathcal{P}_{\theta}\) enables the inner expansion of the solution to achieve an error of order \(O(e^{-\frac{32}{12}s})\), thereby revealing the genuinely nonradial nature of the solution.
\end{remark}

\begin{remark}
    To our knowledge, Theorem \ref{Theorem 1.1} provides a first example of a quenching solution that admits a non-radial profile, given that all existing constructed solutions to \ref{mems} only display radial profiles. For example, in the one-dimensional parabolic local model, Filippas and Guo \cite{filippas1993quenching}, Merle and Zaag \cite{merle1997reconnection} demonstrated that for certain initial data, the corresponding solution quenches solely at the origin with the following asymptotic quenching profile:
\begin{align}
    u(x,T) \sim C\left(\frac{|x|}{\sqrt{|\ln{|x|}|}}\right)^{2/3}, \quad \text{as } |x| \to 0, \quad C>0.\label{profil}
\end{align} 
The same final profile holds in higher-dimensional cases with a nonlocal term constructed in \cite{duong2019profile}, where the constant $C$ is precisely computed. 
\end{remark}

\begin{remark}(Comparison with the semi-linear heat equation)
Quenching solutions for the MEMS model \eqref{mems} share several common features with blowup solutions for the semi-linear heat equation:
\begin{align}
    u_t=\Delta u+|u|^{p-1}u, \quad p \geq 1.\label{nlh}
\end{align}
Indeed, once we make the nonlinear transformation $v = u^{-1}$, the equation \eqref{mems} becomes 
\begin{equation} \label{eq:v}
    \partial_t v = \Delta v + v^4  - 2\frac{|\nabla v|^2}{v}.
\end{equation}
This is the nonlinear heat \eqref{nlh} with an extra transport term $\frac{|\nabla v|^2}{v} = \nabla \ln v  \cdot \nabla v$. This is a problematic term that complicates the analysis due to a possible vanishing of $v$. This is one of the reasons we decided to work on the original form \eqref{mems}.

In the literature, most constructed solutions to the semilinear heat equation possess radial profiles, and considerably less is known about solutions with nonradial behavior. Notable examples of nonradial blow-up solutions to \eqref{nlh} include the work of Del Pino, Musso, Wei, and Zheng \cite{del2018sign}, which builds on explicit constructions of sign-changing nonradial solutions to the associated elliptic problem developed in \cite{del2011large, del2013torus}. Further examples arise in the work of Merle and Zaag \cite{merle2024degenerate}, aligned with the classification results of Velázquez \cite{velazquez1993classification} (see also the earlier works \cite{herrero1993blow, filippas1992refined, filippas1993blowup}), as well as in the anisotropic blow-up solutions constructed by Collot, Merle, and Rapha\"el \cite{collot2017strongly}.

Our construction of a non-radial solution for the MEMS model is inspired by the work of Merle and Zaag \cite{merle2024degenerate}, where they constructed explicit examples of blowup solutions to \eqref{nlh} with a truly non-radial profile. In particular, their analysis relied on a Liouville-type theorem to control \(\nabla \varepsilon\), where \(\varepsilon\) is the perturbation of the solution from the blow-up profile. In our proof, instead of using a Liouville-type theorem, we use a purely energy estimate adapted from the work of Nguyen, Nouaili, and Zaag \cite{nguyen2025construction} to archive the control of \((y \cdot \nabla)^k \varepsilon\), \(k = 0, 1, 2\), in suitable weighted \(H^2\) spaces. This, combined with Sobolev inequalities, enables us to obtain the \(L^\infty\) bound of \(\varepsilon\). We note that by controlling gradient terms in suitable weighted spaces, one can obtain a more direct approach to estimating these terms. This leads to a simplification of the analysis compared to the method used in \cite{merle2024degenerate}. More details are given in the next subsection \ref{Strategy of Analysis} and subsection \ref{subsection5.2}.
\end{remark}

\begin{remark}[Potential Type-II quenching solutions] In this paper, the constructed solution is of Type I singularity in the sense that the quenching solution \(u\) to equation (\ref{mems}) satisfies 
\[
    \liminf_{t \to T} \left( \frac{\min_{x \in \overline{\mathcal{B}}} u(x,t)}{(T-t)^{\frac{1}{3}}} \right) \geq C, \quad \text{for some } C > 0.
\]
Otherwise, \(u\) is said to develop a Type II singularity. Recall that the original equation \eqref{mems} is transformed to \eqref{eq:v} with an extra transport term of the form $\nabla \ln v \cdot \nabla v$, which potentially creates more interesting blowup dynamics in the solution $v$, and for which we expect that equation (\ref{mems}) would possibly exhibit Type II singularities.  
\end{remark}

\subsection{Strategy of analysis}
\label{Strategy of Analysis}
We outline the ideas of the proof of Theorem \ref{Theorem 1.1}, divided in several steps:\\
\textit{– Renormalization in self-similar variables:} Under the transform \ref{selfsimvari},
\(w\) satisfies the self-similar equation,
\begin{align}
    &w_s = \Delta w - \frac{1}{2} y \cdot \nabla w + \frac{1}{3}w - \frac{1}{w^2},\label{eq:wys}\\
    &w(y,s)\equiv e^{\frac{1}{3}s}, \quad \text{ on } |y|=e^{\frac{1}{2}s}. \label{eq:wysboundary}
\end{align}
In this setting, the construction of the quenching profile \(\mathcal{P}_{\theta}\) is performed through a spectral analysis. We remark that the profile $\mathcal{P}_{\theta}$ is connected to the trivial constant solution $\bar w = 3^{\frac{1}{3}}$ in the compact set of $y$-variable. A linearization of \eqref{eq:wys} around $\bar w$ yields a linearized problem driven by the linear Hermite operator $\mathcal{L} =\Delta - \frac{1}2y \cdot \nabla + I$, whose spectral properties are well established. A further refinement leads to a precise approximate solution including the correction $\mathcal{C}(y,s)$ introduced in \eqref{corr}, from which we can determine the profile $\mathcal{P}_{\theta}$, see Section \ref{section2} for a detailed derivation. 

\medskip
\noindent
\textit{– Linearization:} Having derived the good approximate profile $\mathcal{P}_{\theta}$ in the self-similar variables, we want to rigorously prove that the true solution to \eqref{eq:wys} remains close to this profile under small perturbations in some appropriate functional setting. We introduce
\begin{equation}
\label{eq:varep lin}
    w(y,s) = \mathcal{P}_{\theta}(y,s) + \varepsilon(y,s),
\end{equation}
which leads to the linearized evolution equation for \(\varepsilon\):
\begin{equation}
    \partial_s \varepsilon = \mathcal{H} \varepsilon + E + NL(\varepsilon),\label{lineq}
\end{equation}
where \(\mathcal{H}\) is the linearized operator given by
\[
    \mathcal{H} \varepsilon = \Delta \varepsilon - \frac{1}{2} y \cdot \nabla \varepsilon + \frac{1}{3}\varepsilon,
\]
\(E\) is the error generated by the profile \(\mathcal{P}_{\theta}\), and \(NL(\varepsilon)\) the non-linear term.\\
\textit{– Decomposition and control of perturbation \(\varepsilon(y,s)\):} Our aim is to show that \(\|\varepsilon(s)\|_{L^{\infty}(\mathbb{R}^2)} \to 0\) as \(s \to \infty\), thus, we introduce the following:
\begin{align*}
    &\veph(y,s)=\varepsilon(y,s) \cdot \mchi_{K}(z), \quad \eta^{ex}(z,s)=\eta(z,s) \cdot(1-\mchi_{K}(ze^{-\frac{5}{49}s})), \quad z=ye^{-\frac{s}{4}},  
\end{align*}
where \(\eta(z,s)=\varepsilon(y,s)\) and use them to control the perturbation \(\varepsilon\).
\begin{itemize}
    \item Control of \(\varepsilon\) in compact sets:
We further decompose \(\veph\) as
\begin{align*}
    \veph(y,s)=\sum_{i+j \leq 8}\veph_{ij}(s)H_{ij}(y)+\veph_-(y,s),
\end{align*}
where $H_{ij}$ is the eigenfunction of the self-adjoint operator $\mathcal{L}: H^2_\rho(\mathbb{R}^2) \subset L^2_{\rho}(\mathbb{R}^2) \to L^2_{\rho}(\mathbb{R}^2)$
\begin{align}
    \mathcal{L}=\Delta-\frac{1}{2}(y \cdot \nabla)+I\label{def:L},
\end{align}
and $L^2_\rho$ is the weighted $L^2$-space defined in \eqref{def:L2rho}, and corresponds to the eigenvalue $\lambda_{ij} = 1 - \frac{i+j}{2}$, i.e. 
\begin{equation}
    \mathcal{L}(H_{ij}) = \left(1 - \frac{i + j}{2}\right) H_{ij}, \quad H_{ij}(y) = h_i(y_1) h_j(y_2), \label{Hij}
\end{equation}
with $h_j$ being Hermite polynomials defined by \ref{hij}. In particular, the family \(\{H_{ij} : i, j = 0, 1, 2, \dots\}\) forms an orthogonal basis for \(L^2_{\rho}(\mathbb{R}^2)\), and the univariate polynomials satisfy the orthogonality relation
\begin{align}
    \frac{1}{4\pi}\int_{\mathbb{R}} h_n(\xi) h_m(\xi) e^{-\xi^2/4} \, d\xi = 2^n n! \delta_{nm}.\label{otho}
\end{align}
For \(y\) in compact sets,  the equation of $\veph$ roughly reads as 
\begin{align*}
  \partial_s\veph=\mathcal{L}\veph+O(se^{-3s} + |\veph|^2), \quad \text{ in }  L^2_{\rho}(\mathbb{R}^2).
\end{align*}
Thanks to the spectral gap
\begin{align*}
    \left<\mathcal{L}\veph_-,\veph_-\right>_{\rho} \leq -\frac{7}{2}\left<\veph_-,\veph_-\right>_{\rho},
\end{align*}
we obtain the energy estimate (see Lemma \ref{Lemma 5.2})
\begin{align}
        \frac{1}{2}\frac{d}{ds}\|\veph_-(s)\|^2_{L^2_{\rho}(\mathbb{R}^2)} \lesssim -\frac{7}{2}\|\veph_-(s)\|^2_{L^2_{\rho}(\mathbb{R}^2)}+e^{-\frac{32}{6}s}.\label{introen1}
\end{align}
For the finite-dimensional part, projecting the equation onto the eigenfunctions \(H_{ij}\) yields the following ODEs for the coefficients 
 \(\veph_{ij}\) (see Lemma \ref{Lemma 5.3}):
\begin{align*}
        \veph_{ij}'(s)=\left(1-\frac{i+j}{2}\right)\veph_{ij}(s)+O(e^{-\frac{32}{12}s}).
\end{align*}
In particular, for \(i+j=8\), we have
\begin{align}
        |\veph_{ij}'(s)+3\veph_{ij}(s)| \lesssim e^{-\frac{32}{12}s}.\label{introen2}
\end{align}
A forward integration in time of (\ref{introen1}) and (\ref{introen2}) yields the estimates:
\begin{align}
    \|\veph_-(s)\|_{L^2_{\rho}(\mathbb{R}^2)} \lesssim e^{-\frac{32}{12}s}, \qquad |\veph_{ij}(s)| \lesssim e^{-\frac{32}{12}s}, \quad i+j = 8.\label{innerest1}
\end{align}
For the remaining modes with \(i+j\leq 7\), we control them by a topological argument where we specify initial data (see \eqref{initialdata} for a definition) so that these components converge to \(0\) as \(s \to \infty\) with the bound
\begin{align}
    |\veph_{ij}(s)| \lesssim e^{-\frac{32}{12}s}, \quad i+j \leq 7.\label{innerest2}
\end{align}
Finally, estimates \ref{innerest1} and \ref{innerest2} combined with a parabolic regularity, yield control of \(\varepsilon\) over compact sets $\{|y| \leq R\}$ for $R > 0$ arbitrary.\\

\item Control of \(\varepsilon\) in \(1 \lesssim  |y|\lesssim e^{\frac{69}{196}s}\) (equivalently, $e^{-s/4} \lesssim |z| \leq e^{\frac{5}{49}s}$). In this region, thanks to the dissipative structure of the parabolic equation (\ref{lineq}), we can derive the following energy estimates (see Lemmas \ref{Lemma5.5} - \ref{Lemma5.7}, and Lemma \ref{lem:intEngz2}):
\begin{align*}
        &\frac{1}{2}\frac{d}{ds}\|(y \cdot \nabla)^k\varepsilon\|_{\flat}^2 \lesssim-\frac{7}{2}\|(y \cdot \nabla)^k\varepsilon\|^2_{\flat}+O(e^{-\frac{32}{6}s}), \quad k=0,1,2,\\
        &\frac{1}{2}\frac{d}{ds}\|(z \cdot \nabla)^k\eta\|_{\natural}^2 \lesssim-\frac{1}{2}\|(z \cdot \nabla)^k\eta\|^2_{\natural}+O(e^{-\frac{5}{6}s}), \quad k=0,1,2, \quad z=ye^{-s/4},
\end{align*}
where, \(\eta(z,s)=\varepsilon(y,s)\). Moreover, the norms \(\|\cdot\|_{\flat}\) and \(\|\cdot\|_{\natural}\) are given by
\begin{align}
        &\|\varepsilon\|_{\flat}=\left(\int_{|y| \geq K}\frac{|\varepsilon|^2}{|y|^{\alpha}}\,\frac{dy}{|y|^2}\right)^{1/2}, \quad \alpha=18.\\
        &\|\eta\|_{\natural}=\left(\int_{|z| \geq K}\frac{|\eta|^2}{|z|^{\gamma}}\,\frac{dz}{|z|^2}\right)^{1/2}, \quad \gamma=8.01,
\end{align}
where \(K \gg 1\) is a large fixed constant. Integration in time combined with Sobolev inequality yields
\begin{align*}
    &\sup_{|y| \geq K}|\langle y \rangle^{-9} \varepsilon(y,s)| \lesssim e^{-\frac{32}{12}s}, \quad\sup_{|z| \geq K}|\langle z \rangle^{-4.005} \eta(z,s)| \lesssim e^{-\frac{5}{12}s}, 
\end{align*}
which controls the pertabation \(\varepsilon\) in the region \(1 \lesssim  |y|\lesssim e^{\frac{69}{196}s}\).

\item Control of \(\varepsilon\) in \(|z| \gtrsim e^{\frac{5}{49}s}\). We control \(\varepsilon\) in this region by controling \(\eta^{ex}(z,s)\), where \(z=ye^{-\frac{s}{4}}\). By estimating both the generated error and the nonlinear term in \ref{lineq}, we obtain an evolution equation for \(\eta^{ex}\) and derive the bound 
\begin{align*}
    |\eta^{ex}(z,s)| \lesssim e^{-s/6}(1-e^{-\epsilon s})(1-e^{-(|ze^{-s/4}|-1)^2})^2|z|^2,
\end{align*}
by the comparison principle. Here, \(\epsilon>0\) is a small fixed constant, the factor \(1-e^{-\epsilon s}\) is a technical factor to to ensure the comparison principle applies, and the factor \((1-e^{-(|ze^{-s/4}|-1)^2})^2\) is chosen such that \(\eta^{ex} \equiv 0\) for \(|z|=e^{\frac{1}{4}s}\), matching the boundary condition (See Condition \ref{eq:wysboundary}). This completes the estimate for \(\varepsilon\) in the outer region \(|z| \gtrsim e^{\frac{5}{49}s}\).
\end{itemize}
The above decomposition is detailed in the bootstrap regime (Definition \ref{Definition 3,1}), in which we successfully construct solutions admitting the behavior described in Theorem \ref{Theorem 1.1}.

We organize the rest of the paper as follows. In Section \ref{section2}, we formally derive the quenching profile \(\mathcal{P}_{\theta}\). Section \ref{section3} introduces the linearized problem and defines a bootstrap regime to control the perturbation \(\varepsilon(y,s)\). In Section \ref{section4}, we establish key estimates for solutions within the bootstrap regime, which are essential for controlling \(\varepsilon\). Finally, in Section \ref{section5}, we prove that the perturbation \(\varepsilon(y,s)\) remains trapped in the bootstrap regime for all time, thus completing the proof of Theorem \ref{Theorem 1.1}.

\vspace{10pt}

\paragraph{\textbf{Notations:}} Let \(\chi\in C^{\infty}([0,\infty))\) be a smooth cutoff function satisfying
\begin{align*}
    \chi(\xi)=\begin{cases}
            1 \quad \text{ if } \quad 0 \leq\xi \leq 1,\\
            0 \quad \text{ if } \quad \xi \geq 2. 
        \end{cases}
\end{align*}
For \(R > 0\), define \(\mchi_R : \mathbb{R}^2 \to \mathbb{R}\) by
\begin{equation}\label{def:chiR}
    \mchi_R(y) = \chi\left(\frac{|y|}{R}\right).
\end{equation}
We write \( M \lesssim N \) to indicate that there exists a constant \( C > 0 \) such that \( M \leq C N \). Similarly, we write \( M \ll N \) if \( C M \leq N \) for some sufficiently large constant \( C > 1 \).
We denote by \(L^2_{\rho}(\mathbb{R}^2)\) the weighted \(L^2\) space 
\[
L^2_{\rho}(\mathbb{R}^2) = \left\{ f \in L^2_{\text{loc}}(\mathbb{R}^2) : \int_{\mathbb{R}^2} |f(y)|^2 \rho(y)\,dy < \infty \right\},
\]
equipped with the inner product
\begin{equation}\label{def:L2rho}
    \langle f, g \rangle_{\rho} = \int_{\mathbb{R}^2} f(y) g(y) \rho(y)\,dy,
\end{equation}
where the weight function \(\rho(y)\) is 
\[
\rho(y) = \frac{1}{4\pi} e^{-\frac{|y|^2}{4}}.
\]
For derivative notations, let \(f(y,s)\) be a function with \(y=(y_1,y_2)\). We use \(\partial_if\) to denote the partial derivative \(\partial_{y_i}f\). Similarly, if \(g(z,s)\) is a function with \(z=(z_1,z_2)\), we write \(\partial_ig\) for \(\partial_{z_i}g\). In this way, the meaning of \(\partial_i\) is understood from the context of the function and its variables.

\section{Derivation of the Nonradial Quenching Profile}
\label{section2}

To establish Theorem \ref{Theorem 1.1}, we construct an approximate quenching profile valid in different spatial regions. This begins with a detailed analysis of the solution in self-similar variables, focusing first on the inner region near the origin. In this section, we derive the leading-order behavior of the solution and construct the nonradial profile \(\mathcal{P}_{\theta}\) defined in \ref{qp} using spectral analysis.\\

We begin by deriving an approximate solution near the quenching point using the self-similar formulation \ref{selfsimvari}, where the function \(w\) satisfies the self-similar equation
\begin{equation}
    \partial_s w = \Delta w - \frac{1}{2} y \cdot \nabla w + \frac{1}{3}w - \frac{1}{w^2} \quad \text{in } \mathcal{B}_s \times (s_0, \infty).\label{selfsimpde}
\end{equation}
Here, \(\mathcal{B}_s = e^{s/2} \mathcal{B}\), with the boundary condition
\[
    w(y,s) = e^{\frac{s}{3}} \quad \text{for } |y| = e^{s/2}.
\]
We observe that \(\hbar= 3^{\frac{1}{3}}\) is a trivial constant solution to (\ref{selfsimpde}). We then linearize around \(c\) by setting
\[
    w(y,s) = \hbar + \widetilde{w}(y,s),
\]
which yields the evolution equation for \(\widetilde{w}\),
\begin{equation}
    \partial_s \widetilde{w} = \Delta \widetilde{w} - \frac{1}{2} y \cdot \nabla \widetilde{w} + \frac{1}{3}(\hbar+\widetilde{w}) - \frac{1}{(\hbar + \widetilde{w})^2}.
\end{equation}
By expanding the nonlinear term using Taylor expansion:
\[
    \frac{1}{(\hbar+\widetilde{w})^2}=\frac{1}{\hbar^2}\left(1-\frac{2\widetilde{w}}{\hbar}+O(\widetilde{w}^2)\right),
\]
and substituting this into the equation for \(\widetilde{w}\), we obtain
\[
    \partial_s \widetilde{w} = \mathcal{L} \widetilde{w} + O(\widetilde{w}^2),
\]
where \(\mathcal{L}\) is the linearized operator introduced in \eqref{def:L}.

Since the nonlinear remainder is roughly quadratic, the leading-order behavior of \(\widetilde{w}\) is governed by the linear equation
\begin{equation}
    \partial_s \widetilde{w} = \mathcal{L} \widetilde{w}.\label{pdew}
\end{equation}

We note from \eqref{Hij} that \(\widetilde{w} = C e^{-s} H_{22}\) is an exact solution to equation (\ref{pdew}) for arbitrary $C \in \mathbb{R}$. To further refine our quenching profile, we take $C = \frac{\hbar}{9}$ (to be agreed with the chosen profile $\mathcal{P}_{\theta}$) and consider the linearization around 
\[
    \hbar + \frac{\hbar}{9}\, e^{-s} H_{22}(y),
\]
by setting
\begin{equation}
    w(y,s) = \hbar + \frac{\hbar}{9} e^{-s} H_{22}(y) + \Sigma(y,s),\label{innpro}
\end{equation}
where \(\Sigma = o(e^{-s})\) in \(L^2_{\rho}(\mathbb{R}^2)\).
\begin{lemma}
\label{Lemma 2.1}

The correction term \(\Sigma(y,s)\) admits the expansion in $L^2_\rho(\mathbb{R}^2)$,
\begin{align*}
    \Sigma(y,s) = e^{-2s} \Bigg[ 
        &\frac{64}{81\hbar} h_0 h_0 
        + \frac{32}{27\hbar} (h_0 h_2 + h_2 h_0) 
        + \frac{64}{27\hbar} h_2 h_2 
        + \frac{8}{27\hbar} (h_0 h_4 + h_4 h_0) \\
        &- \left( \frac{8}{27 \hbar} s + C_{2,4} \right)(h_2 h_4 + h_4 h_2) 
        + C_{0,6} (h_0 h_6 + h_6 h_0) 
        - \frac{2}{54\hbar} h_4 h_4 
    \Bigg] + O(s^2 e^{-3s}).
\end{align*}
\end{lemma}

\begin{proof}
Substituting equation (\ref{innpro}) into the self-similar equation (\ref{selfsimpde}), we find that \(\Sigma(y,s)\) satisfies the PDE
\begin{equation}
    \Sigma_s = \mathcal{L} \Sigma - \frac{2}{54\hbar} e^{-2s} H_{22}^2 + o(e^{-2s}) 
    \quad \text{in} \quad L^2_{\rho}(\mathbb{R}^2).\label{pdesigma}
\end{equation}
Using the spectral decomposition of \(\mathcal{L}\), we write
\[
    \Sigma(y,s) = \sum_{i,j=0}^{\infty} w_{ij}(s) H_{ij}(y),
\]
and by projecting equation (\ref{pdesigma}) onto each eigenfunction \(H_{ij}\), we derive the following ODE for each coefficient \(w_{ij}(s)\):
\[
    w_{ij}'(s) = \left(1 - \frac{i+j}{2} \right) w_{ij}(s) + \beta_{ij} e^{-2s} + o(e^{-2s}),
\]
where the coefficients \(\beta_{ij}\) are given by
\[
    \beta_{ij} = - \frac{2}{54\hbar} \cdot \frac{
        \langle h_2^2 h_2^2, h_i h_j \rangle_{\rho}
    }{
        \langle h_i h_j, h_i h_j \rangle_{\rho}
    } 
    = -\frac{2}{54\hbar} \cdot \frac{
        \left( \int_{\mathbb{R}} h_2^2 h_i e^{-y_1^2/4} \, dy_1 \right)
        \left( \int_{\mathbb{R}} h_2^2 h_j e^{-y_2^2/4} \, dy_2 \right)
    }{
        \left( \int_{\mathbb{R}} h_i^2 e^{-y_1^2/4} \, dy_1 \right)
        \left( \int_{\mathbb{R}} h_j^2 e^{-y_2^2/4} \, dy_2 \right)
    }.
\]
Solving for \(w_{ij}(s)\) using the orthogonal property \ref{otho} yields the desired expansion for \(\Sigma(y,s)\).
\end{proof}

\medskip

We observe that the constants \(C_{0,6}\) and $C_{2,4}$ appearing in Lemma \ref{Lemma 2.1} are free parameters. We then set $C_{2,4} = 0$ and \(C_{0,6} =\frac{\hbar}{9} \cdot \theta\) with \(0<\theta < \theta^*\) and introduce
\[
    \Psi_{\theta}(y, s) = \Big(3 + e^{-s} y_1^2 y_2^2 + \theta e^{-2s} (y_1^6 + y_2^6) \Big)^{\frac{1}{3}},
\]
which is valid in the domain \(|y| \lesssim e^{s/4} \). Since \(\Psi_{\theta}(y,s)\) is intended to match the profile
\[
    w(y,s) = \hbar + \frac{\hbar}{9} e^{-s} H_{22}(y) + \Sigma(y,s),
\]
we thus introduce a correction to ensure compatibility. Obseving that by expanding \(\Psi_{\theta}(y,s)\), we obtain
\begin{align*}
    \Psi_{\theta}(y,s) & = \Big(3 + e^{-s} y_1^2 y_2^2 + \theta e^{-2s} (y_1^6 + y_2^6) \Big)^{\frac{1}{3}} \\
    &= \hbar + \frac{\hbar}{9} e^{-s} y_1^2 y_2^2 
       + \frac{\hbar}{9} \cdot \theta e^{-2s} (y_1^6 + y_2^6) 
       - \frac{\hbar}{81} e^{-2s} y_1^4 y_2^4 
       + O(e^{-3s} |y|^{10}), \quad \forall |y| \lesssim e^{s/4}.
\end{align*}
Using this expansion, we define the modified profile

\[
    \Psi_{\theta}^{\mathrm{mod}}(y,s) = \left(3 + e^{-s} y_1^2 y_2^2 + \theta e^{-2s} (y_1^6 + y_2^6)\right)^{\frac{1}{3}} + \mathcal{C}(y,s)\mchi_K(z),
\]
where \(\mathcal{C}(y,s)\) is the correction term \ref{corr}. Note that by adding the correction term \(\mathcal{C}(y,s)\) to \(\Psi_{\theta}(y,s)\), the expansion of the leading terms aligns with that of \(w(y,s)\). This confirms that the modified profile \(\Psi_{\theta}^{\mathrm{mod}}(y,s)\) is indeed compatible with the inner expansion \(w(y,s)\), and remains valid within the region \( |y| \lesssim e^{s/4}\).

To construct an approximate profile in the outer region \(|y|\gtrsim e^{s/4}\), we introduce 
\[
     v(z,s) = w(y, s), \quad z = y e^{-s/4}.
\]
The function \(v\) satisfies the equation
\begin{equation}
    \partial_s v = e^{-s/2} \Delta v - \frac{1}{4} z \cdot \nabla v + \frac{1}{3}v - \frac{1}{v^2} \quad \text{in } \mathcal{D}_s \times (s_0, \infty),
\end{equation}
where \(\mathcal{D}_s = e^{s/4} \mathcal{B}\), with boundary condition
\begin{equation}
    v(z,s) = e^{\frac{s}{3}}, \quad \text{on } |z| = e^{s/4}.\label{bddz}
\end{equation}
Since the term \(e^{-s/2} \Delta v\) becomes negligible as \(s \to \infty\), we consider the simplified equation
\begin{equation} \label{eq:vzs}
    \partial_s v = -\frac{1}{4} z \cdot \nabla v + \frac{1}{3}v - \frac{1}{v^2},
\end{equation}
in the region \(|z| \gtrsim 1\). We note that \(v(z) = (3 + z_1^2 z_2^2)^{\frac{1}{3}}\) is an exact solution of this equation, and to ensure compatibility with the inner profile \(\Psi_{\theta}^{\text{mod}}\) near \(|z| \sim 1\), and to match the boundary behavior from (\ref{bddz}), we slighly modify $\Psi_\theta^{\textup{mod}}$ that leads us to the fully approximated quenching profile \(\mathcal{P}_{\theta}\) defined in \ref{qp}.

\section{Linearization, error decomposition and bootstrap definition}
\label{section3}
With the approximate profile \(\mathcal{P}_{\theta}\) defined as in \eqref{qp}, the next step is to rigorously justify that the true solution remains close to $\mathcal{P}_{\theta}$. To do so, we linearize the equation \eqref{eq:wys} around $\mathcal{P}_{\theta}$, i.e. $w = \mathcal{P}_{\theta} + \varepsilon$, and control the perturbation \(\varepsilon(y,s)\) in a certain bootstrap regime. To be concrete, we show that there exist smooth initial data \(u_0\) such that the corresponding solution \(u\), expressed in self-similar variables, satisfies the convergence
\begin{align*}
    \sup_{0 \leq |y| \leq e^{\frac{69}{196}s}}|w(y,s)-\mathcal{P}_{\theta}(y,s)| \to 0, \quad \sup_{e^{\frac{69}{196}} \leq |y| \leq e^{\frac{1}{2}s}}\left|\frac{w(y,s)-\mathcal{P}_{\theta}(y,s)}{\mathcal{P}_{\theta}(y,s)}\right| \ll 1, \quad \text{as } s \to \infty
\end{align*}
(The technical factor $\frac{69}{196} = \frac{1}{4} + \frac{5}{49}$ is chosen to simplify the nonlinear estimate in the region $|z| \geq e^{\frac{5}{49}s}$, see Lemmas \ref{Lemma5.5}-\ref{Lemma5.7} and Lemma \ref{lem:intEngz2}). The linearization around \(\mathcal{P}_{\theta}\) leads to the evolution equation for the perturbation \(\varepsilon\):
\begin{align}    \partial_s\varepsilon=\mathcal{H}\varepsilon+E+NL(\varepsilon),\label{semilinearpde}
\end{align}
where each term is defined as follows:
\begin{align}
    &\mathcal{H}\varepsilon=\Delta \varepsilon-\frac{1}{2}y \cdot \nabla \varepsilon+\frac{1}{3}\varepsilon \label{H}\\
    &E=-\partial_s\mathcal{P}_{\theta}+\Delta \mathcal{P}_{\theta}-\frac{1}{2}y \cdot \nabla \mathcal{P}_{\theta}+\frac{1}{3}\mathcal{P}_{\theta}-\frac{1}{\mathcal{P}_{\theta}^2},\label{E}\\
    &NL(\varepsilon)=\frac{1}{\mathcal{P}_{\theta}^2}-\frac{1}{(\mathcal{P_{\theta}}+\varepsilon)^2}.\label{NL}
\end{align}
Our main objective is to show that there is an initial data $\varepsilon_0 \in L^\infty(\mathbb{R}^2)$ such that the corresponding solution to \eqref{semilinearpde} is global-in-time and satisfies
\(\|\varepsilon(s)\|_{L^{\infty}(\mathbb{R}^2)} \to 0\) as \(s \to \infty\). For this purpose, we introduce the decomposition
\begin{align}
    &\veph(y,s)=\varepsilon(y,s) \cdot \mchi_{K}(z),\label{eq:decomin}\\
    &\eta^{ex}(z,s)=\eta(z,s) \cdot(1-\mchi_K(ze^{-\frac{5}{49}s})), \quad z=ye^{-\frac{s}{4}}, \label{eq:outdec}
\end{align}
where \(\eta(z,s)=\varepsilon(y,s)\) and \(K \gg 1\) is a large fixed constant. We further decompose \(\veph\) as 
\begin{align}
    \veph(y,s)=\sum_{i+j \leq 8}\veph_{ij}(s)H_{ij}(y)+\veph_-(y,s),\label{decomoin} \quad \quad \veph_- \perp_{L^2_\rho} \{H_{ij}\}_{i+j \leq 8},
\end{align}
where the eigenfunctions \(H_{ij}\) are defined in \ref{Hij}. Note that we have the spectral gap estimate
\begin{equation}
\label{inq:spec gap}
    \langle \mathcal{L} \veph_-, \veph_- \rangle_{L^2_\rho} \leq  - \frac{7}{2}\langle\veph_-,\veph_-\rangle_{\rho},
\end{equation}
and by integration by parts, we have the orthogonality conditions
\begin{align}
    & \partial_k\veph \perp_{L^2_\rho} \{H_{ij}\}_{i+j \leq 7}, \quad k=1,2, \label{con:orth 1}\\
    & \partial^2_{kl}\veph \perp_{L^2_\rho} \{H_{ij}\}_{i+j \leq 6}, \quad k,l=1,2,\label{con:orth 2}
\end{align}
which leads to the spectral gap estimates
\begin{align}
    &\left<\mathcal{L} \partial_k \veph_-,\partial_k \veph_-\right>_{\rho} \leq -3\left<\partial_k \veph_-,\partial_k \veph_-\right>, \quad k=1,2,\label{inq:spec partial i}\\
    &\left<\mathcal{L}\partial^2_{kl}\varepsilon_-,\partial^2_{kl}\varepsilon_-\right>_{\rho} \leq -\frac{5}{2}\left<\partial^2_{kl}\varepsilon_-,\partial^2_{kl}\varepsilon_-\right>_{\rho}, \quad k,l=1,2. \label{inq:spec partial ij}
\end{align}
To fully control the perturbation
\(\varepsilon(y,s)\) and close the nonlinear estimate, we introduce the following bootstrap regime.

\begin{definition}
\label{Definition 3,1}
Let $A_1, A_2.A_3, B_1, B_2, B_3, C_1, C_2, C_3, D_1$ and $s$ be positive, we denote by \(S(s)\) the set of functions \(\varepsilon(y,s)\) such that the decomposition (\ref{eq:decomin}) and (\ref{eq:outdec}) satisfy \\
(i) (Inner control)
\begin{align}
    &\left( \sum_{i+j \leq 7} |\veph_{ij}(s)|^2 \right)^{\frac{1}{2}} \leq A_1 e^{-\frac{32}{12}s},\label{in1}\\
    &|\veph_{ij}(s)| \leq A_1e^{-\frac{32}{12}s}, \text{ for } i+j =8, \label{in2}\\
    &\|\veph_-(s)\|_{L^2_{\rho}(\mathbb{R}^2)} \leq A_1 e^{-\frac{32}{12}s} \label{in3},\\
    &\|\partial_i\veph_-(s)\|_{L^2_{\rho}(\mathbb{R}^2)} \leq A_2 e^{-\frac{32}{12}s}, \quad i=1,2,\label{in4}\\
    &\|\partial^2_{ij}\veph_-(s)\|_{L^2_{\rho}(\mathbb{R}^2)} \leq A_3 e^{-\frac{32}{12}s}, \quad i,j=1,2, \label{in5}
\end{align}
(ii) (Intermediate control)
\begin{align}
    &\|(y \cdot \nabla)^k\varepsilon\|_{\flat} \leq B_{k+1} e^{-\frac{32}{12}s}, \quad k=0,1,2,\label{int1}\\
   \quad  &\|(z \cdot \nabla)^k \eta \|_{\natural} \leq C_{k+1} e^{-\frac{5}{12}s}, \label{bootstrap: int 2} \quad k=0,1,2,
\end{align}
where \(\eta(z,s)= \varepsilon(y,s)\). Moreover, the norms \(\
\|\varepsilon\|_{\flat}\) and \(\
\|\eta\|_{\natural}\) is defined by
\begin{align}
        &\|\varepsilon\|_{\flat}=\left(\int_{|y| \geq K}\frac{|\varepsilon|^2}{|y|^{\alpha}}\,\frac{dy}{|y|^2}\right)^{\frac{1}{2}}, \quad \alpha=18,\label{int2}\\
        &\|\eta\|_{\natural}=\left(\int_{|z| \geq K}\frac{|\eta|^2}{|z|^{\gamma}}\,\frac{dy}{|z|^2}\right)^{\frac{1}{2}}, \quad \gamma=8.01\label{bootstrap: int norm 2},
\end{align}
and \(K\gg 1\) is the large fixed constant in \ref{eq:decomin} and \ref{eq:outdec}.\\
(iii) (Outer control)
\begin{align}
       &|\eta^{ex}(z,s)| \leq D_1h(z,s), \quad h(z,s)=\epsilon e^{-s/6}(1-e^{-\epsilon s})(1-e^{-(|ze^{-s/4}|-1)^2})^2|z|^2,\label{out2}
\end{align}
where \(\epsilon>0\) is a small fixed constant such that \(|\eta^{ex}(z,s)/\mathcal{P}_{\theta}(z,s)| \ll 1\). This choice of \(\epsilon>0\) simplifies the bound for the outer nonlinear estimate (See Lemma \ref{Lemma4.1}(iv)).
\end{definition}
\begin{remark}The bootstrap constants are chosen in the order 
\begin{align*}
     D_1 \gg C_3 \gg C_2 \gg C_1 \gg B_3\gg B_2\gg B_1\gg A_3\gg A_2 \gg A_1 \gg 1.
\end{align*}
\end{remark}

 \subsection{Existence of Solutions in the Bootstrap Regime}
We show that for suitable initial data \(\varepsilon_0(y) \in S(s_0)\), the corresponding solution
\(\varepsilon(y,s)\) of \ref{semilinearpde} remains trapped in \(S(s)\) for all \( s \geq s_0\), thereby proving Theorem \ref{Theorem 1.1}. For fixed \(s_0=s_0(A_1, A_2.A_3, B_1, B_2, B_3, C_1, C_2, C_3, D_1,K) \gg 1\), we consider initial data of the form
\begin{align}
    \varepsilon_0(y)= \psi_{s_0}[\textbf{c}](y)=\eta_0(z)&=e^{-\frac{32}{12}s_0}\left(\sum_{i+j \leq 7}c_{ij}H_{ij}(y)\right)\chi(z), \quad \mathbf{c} = (c_{ij})_{0 \leq i+j \leq 7}.\label{initialdata}
\end{align}
Here, \(\chi(z):=\mchi_R(z)\) with \(R=1\) (see \ref{def:chiR}). We now claim the following.
\begin{lemma}
\label{lem:initial data bound} The initial data of \eqref{eq:wys} is of the form
\[
w_0(y) = \mathcal{P}_{\theta}(z, s_0) + \psi_{s_0}[\mathbf{c}](y),
\]
and satisfies \(0<w_0(y) \leq e^{\frac{s_0}{3}}\) for all \(|y| \leq e^{\frac{s_0}{2}}\) and \(w_0(y)=e^{\frac{s_0}{3}}\) on \(|y|=e^{\frac{s_0}{2}}\).
\end{lemma}

\begin{proof}
For \(0 \leq |y| \leq 2Ke^{\frac{s_0}{4}}\), by the definition of \(\mathcal{P}_{\theta}\) in \ref{qp}, we have
\begin{align*}
    0 <w_0(y)&=\left(3 + z_1^2 z_2^2 + \theta e^{-\frac{s_0}{2}}(z_1^6 + z_2^6) \right)^{\frac{1}{3}}\mchi_{\frac{1}{4}\kappa(s_0)}(z) + \mathcal{C}(y,s) \mchi_K(z)+e^{-\frac{32}{12}s_0}\left(\sum_{i+j \leq 7}c_{ij}H_{ij}(y)\right)\chi(z)\\
    & \lesssim \left(3 + \frac{1}{2}|z|^4 + \theta e^{-\frac{s_0}{2}}|z|^6\right)^{\frac{1}{3}}\mchi_{\frac{1}{4}\kappa(s_0)}(z)+s_0e^{-\frac{s_0}{2}}+e^{-\frac{11}{12}s_0} \lesssim 1.
\end{align*}
For \(2Ke^{\frac{s_0}{4}} \leq |y| \leq \frac{1}{2}e^{\frac{s_0}{2}}\), since \(0<\theta<\theta^*\) for some \(\theta^*>0\),\\
we have
\(\left(3 + z_1^2 z_2^2 + \theta e^{-s/2}(z_1^6 + z_2^6)\right)^{\frac{1}{3}}\mchi_{\frac{1}{4}\kappa(s_0)}(z)<e^{\frac{s_0}{3}}\), and hence
\begin{align*}
    0<w_0(y) &= \left(3 + z_1^2 z_2^2 + \theta e^{-\frac{s_0}{2}}(z_1^6 + z_2^6) \right)^{\frac{1}{3}}\mchi_{\frac{1}{4}\kappa(s_0)}(z) +e^{\frac{s_0}{3}}\mchi_{\frac{1}{4}\kappa(s_0)}(z)\\
    & \leq e^{\frac{s_0}{3}}\mchi_{\frac{1}{4}\kappa(s_0)}(z)+e^{\frac{s_0}{3}}\left(1-\mchi_{\frac{1}{4}\kappa(s_0)}(z)\right) \leq e^{\frac{s_0}{3}}.
\end{align*}
Finally, for \(\frac{1}{2}e^{\frac{s_0}{2}} \leq |y| \leq e^{\frac{s_0}{2}}\), we have by definition
\begin{align*}
    w_0(y)=e^{\frac{s_0}{3}}.
\end{align*}
This concludes the proof of Lemma \ref{lem:initial data bound}.
\end{proof}

\begin{proposition}
\label{proposition3.1}
Suppose \(|c_{ij}|<1\), then 
\(\varepsilon_0(y)\in S(s_0)\) with strict inequalities.
\end{proposition}

\begin{proof}
    First, we check the bootstrap conditions (\ref{in1})-(\ref{in5}). Note that we have the relations
    \begin{align*}
        \psi_{s_0}[\textbf{c}](y) \cdot \mchi_K(z)&=e^{-\frac{32}{12}s_0}\left(\sum_{i+j \leq 7}c_{ij}H_{ij}(y)\right)\chi(z) =\sum_{i+j \leq 8}\veph_{ij}(s_0)H_{ij}(y)+\veph_-(y,s_0).
    \end{align*}
    Thus, for \(k+l \leq 7\), we have 
        \begin{align*}
        \varepsilon_{kl}(s_0)&=\left<\psi_{s_0}[\textbf{c}]\cdot \mchi_K(z),H_{kl}\right>_{\rho}/\left<H_{kl},H_{kl}\right>_{\rho}\\
        &=\left<e^{-\frac{32}{12}s_0}\left(\sum_{i+j \leq 7}c_{ij}H_{ij}\right)\chi(z),H_{kl}\right>_{\rho}/\left<H_{kl},H_{kl}\right>_{\rho}\\
        &=\sum_{i+j \leq 7}c_{ij}e^{-\frac{32}{12}s_0}\left(\frac{\left<H_{ij}\chi(z),H_{kl}\right>_{\rho}}{\left<H_{kl},H_{kl}\right>_{\rho}}\right) = c_{kl}e^{-\frac{32}{12}s_0}\big(1 + O(e^{-100 s_0}) \big).
    \end{align*}
This  deduces
\begin{align*}
        &\left(\sum_{i+j \leq 7}|\veph_{ij}(s_0)|^2\right)^{\frac{1}{2}} \ll A_1e^{-\frac{32}{12}s_0},       
\end{align*}
for \(A_1 \gg 1\). Clearly, we have \(|\veph_{ij}(s_0)| \ll A_1e^{-\frac{32}{12}s_0}\) for \(i+j=8\), which shows the bootstrap condition \ref{in2} hold for \(s=s_0\) with a strict inequality. As for the bootstrap condition \ref{in3}, we have
    \begin{align*}
        \veph_-(y,s_0)&=\psi_{s_0}[\textbf{c}](y)\cdot \mchi_{K}(z)-\sum_{i+j \leq 8}\veph_{ij}(s_0)H_{ij}\\
        &=\sum_{i+j \leq 7}(e^{-\frac{32}{12}s_0}c_{ij}\chi(z)-\veph_{ij}(s_0))H_{ij}(y)-\sum_{i+j=8}\veph_{ij}(s_0)H_{ij},
    \end{align*}
which implies
\begin{align*}
    \|\veph_-(s_0)\|_{L^2_{\rho}(\mathbb{R}^2)}\ll A_1e^{-\frac{32}{12}s_0},
\end{align*}
and hence the bootstrap condition (\ref{in3}) holds for \(s=s_0\) with a strict inequality. By differentiating the expression of \(\veph_-\), we derive that

\begin{align*}
    \|\partial_i\veph_-(s_0)\|_{L^2_{\rho}(\mathbb{R}^2)}  \ll A_2 e^{-\frac{32}{12}s_0}, \quad i=1,2,
\end{align*}
for \(A_2 \gg A_1\) and  
\begin{align*}
     \|\partial^2_{ij}\veph_-(s_0)\|_{L^2_{\rho}(\mathbb{R}^2)}  \ll A_3 e^{-\frac{32}{12}s_0}, \quad i,j=1,2,
\end{align*}
for \(A_3 \gg A_2\). This shows the bootstrap condition \ref{in4} and \ref{in5} hold with a strict inequality at \(s=s_0\), respectively. Next, by a direct calculation, we have \(\|(y \cdot \nabla)^k \varepsilon_0\|_{\flat} \ll B_{k+1}e^{-\frac{32}{12}s_0}\) and \(\|(z \cdot \nabla)^k \eta_0\|_{\flat} \ll C_{k+1}e^{-\frac{5}{12}s_0}\), thus proving the bootstrap condition \ref{int1} and \ref{bootstrap: int 2} is valid at \(s=s_0\) with a strict inequality. Finally, the bootstrap condition \ref{out2} holds with a strict inequality at \(s=s_0\) since \(|\eta_0(z)| \equiv
0\) for \(|z| \geq Ke^{\frac{5}{49}s}\).
\end{proof} 

By proposition \ref{proposition3.1}, given initial data of the form 
 \(\varepsilon_0=\psi_{s_0}[\textbf{c}]\), with \(|c_{ij}|<1\), the local Cauchy problem (see \cite{kavallaris2018non}) ensures the existence of a solution  \(\varepsilon(y,s)\in S(s)\) for all \(s \in [s_0,s^*]\). Thus, theorem \ref{Theorem3.1} follows if we can show that 
\(s^*=+\infty\), which we claim in the following.

\begin{theorem}
\label{Theorem3.1} There are $\mathbf{c} = (c_{ij})_{0 \leq i+j \leq 7}$ with \(|c_{ij}|<1\) such that the solution $\varepsilon(y,s)$ to \eqref{semilinearpde} with the intial data \(\varepsilon_0=\psi_{s_0}[\textbf{c}]\), satisfies \(\varepsilon(s) \in S(s)\) for all \(s \in [s_0,\infty)\).
\end{theorem}
The proof of Theorem \ref{Theorem3.1} follows directly from the two following propositions.
\begin{proposition}[Reduction to a finite dimensional problem]
\label{Proposition3.2}
For any \(s_1 \gg s_0\), let \(\varepsilon(y,s)\in S(s)\) for all \(s \in [s_0,s_1]\), then all the bootstrap estimates in Definition \ref{Definition 3,1} hold for all \(s \in [s_0,s_1]\) with strict inequalities except possibly for (\ref{in1}).
\end{proposition}

\begin{proposition}
\label{Proposition3.3}
There are $\mathbf{c} = (c_{ij})_{0 \leq i+j \leq 7}$ with \(|c_{ij}|<1\) such that the solution $\varepsilon(y,s)$ to \eqref{semilinearpde} with the intial data \(\varepsilon_0=\psi_{s_0}[\textbf{c}]\) satisfies the bootstrap estimates (\ref{in1}) for all $s \geq s_0$. 
\end{proposition}
    
The proof of Proposition \ref{Proposition3.2} is technical and lengthy, which is postponed later. We first assume Proposition \ref{Proposition3.2} and proceed with the proof of Proposition \ref{Proposition3.3}.\\

\paragraph{\textbf{Proof of Proposition \ref{Proposition3.3}:}} We proceed by contradiction. Suppose there exists \(\bar{s} \gg s_0\) such that for every solution \(\varepsilon(y,s)\) corresponding to initial data \(\psi_{s_0}[\mathbf{c}]\), with \(|c_{ij}| < 1\), the bootstrap condition \eqref{in1} hold with equality at \(s=\bar s\). In this case, we refer to the time \(\bar{s}\) as the exit time. By continuity, it follows that
\begin{align}
    \left( \sum_{i+j \leq 7} |\veph_{ij}(\bar{s})|^2 \right)^{\frac{1}{2}} = A_1 e^{-\frac{32}{12}\bar{s}}, \label{contrahypo}
\end{align}
moreover, we also have the strictly outgoing property:
\begin{align}
    \frac{1}{2} \frac{d}{ds} \left( \sum_{i+j \leq 7} \left| \frac{e^{\frac{32}{12}\bar{s}}}{A_1} \veph_{ij}(\bar{s}) \right|^2 \right) > 0, \label{strictly}
\end{align}
whose proof is referred to the end of Subsection~\ref{Control of veph in K}.
Let \(\mathcal{U}\) be the open unit disc, and define the map \(\Theta : D(\Theta) \subset \mathcal{U} \to \partial \mathcal{U}\) by
\[
\frac{e^{\frac{32}{12}s_0}}{A_1}\big(\veph_{ij}(s_0)\big)_{0 \leq i+j \leq 7} \mapsto \frac{e^{\frac{32}{12}\bar{s}}}{A_1}\big(\veph_{ij}(\bar{s})\big)_{0 \leq i + j \leq 7},
\]
where the domain \(D(\Theta)\) is defined as
\[
D(\Theta) = \left\{ \frac{e^{\frac{32}{12}s_0}}{A_1}\big(\veph_{ij}(s_0)\big)_{0 \leq i+j \leq 7} \in \mathcal{U} : \text{the exit time } \bar{s} < \infty \right\}.
\]
Note that we may choose $\mathbf{c} = (c_{ij})_{0 \leq i+j \leq 7}$ such that \(\frac{e^{\frac{32}{12}s_0}}{A_1}(\veph_{ij}(s_0))_{0 \leq i+j \leq 7} \in \partial \mathcal{U} \). By the strictly outgoing property, \(\sum_{0 \leq i+j \leq 7}\frac{e^{\frac{32}{12}s}}{A_1}|\veph_{ij}(s)|^2>1\) for \(s>s_0\), thus \(\bar s=s_0\), which implies that \(\Theta\) is the identity on \(\partial \mathcal{ U}\). Next, since the unstable modes \(\veph_{ij}(s)\) depends continuously on the parameter \(\textbf{c}\) and the vector field \(\frac{e^{\frac{32}{12}s}}{A_1}(\veph_{ij}(s))_{0 \leq i+j \leq 7}\) satisfies the strictly outgoing property, \(D(\theta)\) is open and \(\Theta\) is continuous on \( D(\Theta)\). Under the contradiction assumption, combinded with the strictly outgoing property, every initial datum admits a finite exit time, hence \(D(\Theta)=\mathcal{U}\), and thus \(\Theta\) defines a continuous map from \(\bar{\mathcal{U}}\) to \( \bar{\partial\mathcal{U}}\) which is the identity on \( \partial\mathcal{U}\). This contradicts the Brouwer fixed-point theorem.
\qed

\vspace{10pt}

\paragraph{\textbf{Proof of Theorem \ref{Theorem 1.1}:}} Note that Proposition \ref{Proposition3.2} together with Proposition \ref{Proposition3.3} directly yields Theorem \ref{Theorem3.1}. Let \(u\)
denote the solution constructed in Theorem \ref{Theorem3.1}. Since properties (i), (ii), and (iii) of Theorem \ref{Theorem 1.1} follow immediately from Theorem \ref{Theorem3.1}, it remains to establish property (iv). The argument used by Merle in \cite{merle1992solution} can be adapted to this case to get the existence of a function \(u^*\in C(\mathbb{R}^2\setminus \{0\})\) such that \(u \to u^*\) uniformly on compact subsets of \(\mathbb{R}^2\setminus \{0\}\) as \(t \to T\). We now determine the asymptotic behavior of \(u^*\) as \(x \to 
0\). Given \(K'>0\), by property (i), we have
\begin{align*}
    u(e^{-\frac{s}{2}}y,T-e^{-s}) \to e^{-\frac{s}{3}}\left(3+e^{-s}y_1^2y_2^2+\theta e^{-2s}(y_1^6+y_2^6)\right
    )^{\frac{1}{3}},
\end{align*}
uniformly on \(|y| \leq K'\) as \(s \to \infty\). Hence, given \(\epsilon>0\), there is \(s_{\epsilon} \gg 1\) such that for all \(s \geq s_{\epsilon}\),
\begin{align*}
    \left|u(e^{-\frac{s}{2}}y,T-e^{-s})-e^{-\frac{s}{3}}\left(3+e^{-s}y_1^2y_2^2+\theta e^{-2s}(y_1^6+y_2^6)\right)^{\frac{1}{3}}\right|<\epsilon,
\end{align*}
for all \(|y| \leq K'\). Then for \(|x|\) small, \(x \neq 0\), we have, in the \(x\)-variables
\begin{align*}
    \left|u^{*}(x)-\left(x_1^2x_2^2+\theta(x_1^6+x_2^6)\right)^{\frac{1}{3}}\right| \leq \epsilon,
\end{align*}
showing the asymptotic behavior 
\begin{align*}
     u^{*}(x)\sim\left(x_1^2x_2^2+\theta(x_1^6+x_2^6)\right)^{\frac{1}{3}} \quad \text{ as } \quad |x|  \to  0. 
\end{align*}
This completes the proof of Theorem \ref{Theorem 1.1}, assuming Proposition \ref{Proposition3.2}, whose proof is given in section \ref{section5}. 
\qed

 \section{Properties of Solutions in \(S(s)\),  and the Decomposition of the Generated Error \(E\)}
 \label{section4}
 
The success of the bootstrap argument depends on precise estimates of the error term \ref{E}, the nonlinear term \ref{NL}, and the behavior of the solution within the bootstrap regime. In this section, we establish these estimates, focusing on the properties of solutions in the set \(S(s)\) and the detailed structure of the generated error term \(E\) and the nonlinear term \(NL(\varepsilon)\).

\subsection{Properties of \(S(s)\)}
\begin{lemma}
\label{Lemma4.1}
    Let \(s_0 = s_0(A_1, A_2,A_3,K,B_1,B_2,B_3,C_1,C_2,C_3,D_1) \gg 1\), let \(\varepsilon(s) \in S(s)\) be a solution for \(s \in [s_0, s_1]\), with fixed \(s_1 > s_0\). Then we have the following estimates:\\
(i) \textup{(Local \(L^{\infty}\)-estimate for \(\veph_-\).) }
    \begin{equation*}
        \|\veph_-(s)\|_{L^{\infty}(|y| \leq 100K)} \lesssim C(K) \, e^{-\frac{32}{12}s},
    \end{equation*}
    where \(C(K)>0\) is a constant depending only on \(K\) and \(\veph_-\) is defined in \ref{decomoin}.\\
(ii) \textup{(Pointwise estimates for \(\varepsilon\) and \(\eta\)}.) 
We have the pointwise estimates
\begin{align*}
    |\varepsilon(y,s)| \lesssim e^{-\frac{32}{12}s}|y|^9, \text{ for } |y| \geq 2K, \quad |\eta(z,s)| \lesssim e^{-\frac{5}{12}s}|z|^{4.005}, \text{ for } |z| \geq 2K.
\end{align*}
(iii) \textup{(Nonlinear estimate for \(\veph\) in \(L^2_{\rho}(\mathbb{R}^2)\).)} 
The nonlinear term corresponding to \(\veph\) is defined as
\begin{align*}
        NL(\veph)=\sum_{k=2}^{\infty}\binom{-2}{k}\frac{\veph^k}{\mathcal{P}_{\theta}^{k+2}},
    \end{align*}
    where 
    \begin{align*} \binom{-2}{k}=\frac{(-2) \cdot (-3) \cdots(-2-k+1)}{k!}, \quad k \geq 2,
    \end{align*}
    and satisfies the estimate
    \begin{align*}
        NL(\veph)=o(se^{-3s}), \quad \text{ in } L^2_{\rho}(\mathbb{R}^2).
\end{align*}
(iv) \textup{(Exterior nonlinear estimate)}
The exterior nonlinear term is defined as
\begin{align*}
    NL^{ex}(\eta)=NL(\eta)\left(1-\mchi_{K}(ze^{-\frac{5}{49}s})\right),
\end{align*}
where
\begin{align*}
    NL(\eta):=\left(\frac{1}{\mathcal{P}_{\theta}^2}-\frac{1}{(\mathcal{P}_{\theta}+\eta)^2}\right),
\end{align*}
and satisfies the estimate
\begin{align*}
    |NL^{ex}(\eta)| \lesssim e^{-\frac{11}{147}s}
\end{align*}
\end{lemma}
\begin{proof}(i) Note that by Sobolev inequality, we have
    \begin{align*}
      \|\veph_-(s)\|^2_{L^{\infty}(|y| \leq 100K)} \leq C\left(\|D^2 \veph_-(s)\|^2_{L^2(|y| \leq 50K)}+\|D \veph_-(s)\|^2_{L^2(|y| \leq 50K)}+\|\veph_-(s)\|^2_{L^2(|y| \leq 50K)}\right),
    \end{align*}
    where \(C>0\) is the Sobolev constant. By the bootstrap condition (\ref{in3}), we have:
   \begin{align*}
    \|\veph_-(s)\|_{L^2(|y| \leq 50K)}& \lesssim C(K) \|\veph_-\|_{L^2_{\rho}(\mathbb{R}^2)} \lesssim C(K)\,e^{-\frac{32}{12}s}.
   \end{align*}
Next, by the bootstrap condition \ref{in4} and \ref{in5}, we have
   \begin{align*}
      &\|D \veph_-(s)\|_{L^2(|y| \leq 50K)} \lesssim C(K)\,e^{-\frac{32}{12}s}, \quad \|D^2 \veph_-(s)\|_{L^2(|y| \leq 50K)} \lesssim C(K)\,e^{-\frac{32}{12}s},
   \end{align*}
hence, we have
   \begin{align*}
      \|\veph_-(s)\|_{L^{\infty}(|y| \leq 100 K)} \lesssim  C(K)\,e^{-\frac{32}{12}s}.
   \end{align*} 
This concludes the local \(L^{\infty}\) estimate of \(\veph_-\).\\
 (ii) This follows directly from bootstrap conditions \ref{int1} and \ref{bootstrap: int 2} combined with Sobolev inequality.\\
(iii)
Note that by (ii), we have
\begin{align}
\label{inq:ptwise1}
    |\veph(y,s)| \leq M e^{-\frac{32}{12}s}|y|^9, \quad |y| \geq 2K,
\end{align}
moreover, by the local \(L^{\infty}\) estimate for \(\veph_-\), together with the bootstrap condition \ref{in1} and \ref{in2}, we have
\begin{align}
\label{ptwise local veph 0}
    \|\veph(y,s)\|_{L^{\infty}(|y|\leq 50K)} \leq Me^{-\frac{32}{12}s}.
\end{align}
for some constant \(M=M(A_1,A_2,A_3,K,B_1,B_2,B_3)>0\). Thus, we have, for \(k \geq 2\),
    \begin{align*}
        \|-(\veph^k/\mathcal{P}_{\theta}^{k+2})\|_{L^2_{\rho}(\mathbb{R}^2)}& \leq \|\veph^k\|_{L^2_{\rho}(\mathbb{R}^2)}\\
        &=\left(\int_{|y| \leq 2K}|\veph|^{2k}\rho(y)\,dy+\int_{|y| \geq 2K}|\veph|^{2k}\rho(y)\,dy\right)^{1/2} \leq 2M^k\,e^{-\frac{32}{12}ks},
    \end{align*}
    and hence
    \begin{align*}
        \|NL(\veph)\|_{L^2_{\rho}(\mathbb{R}^2)} &\leq \sum_{k=2}^{\infty}(k+1) \|\veph^k\|_{L^2_{\rho}(\mathbb{R}^2)} \leq\sum_{k=2}^{\infty}(k+1)(2M^ke^{-\frac{32}{12}ks})\ll se^{-3s}.
    \end{align*}
    This proves the nonlinear estimate for \(\veph\) in \(L^2_{\rho}(\mathbb{R}^2)\).\\
   (iv) Note that 
    \begin{align*}
        |NL^{ex}(\eta)| \lesssim\left( \frac{1}{\mathcal{P}_{\theta}^2}+\frac{1}{(\mathcal{P}_{\theta}-|\eta|)^2}\right)\left(1-\mchi_{K}(ze^{-\frac{5}{49}s})\right)\lesssim \frac{1}{\mathcal{P}_{\theta}^2}\left(1-\mchi_{K}(ze^{-\frac{5}{49}s})\right) \lesssim e^{-\frac{11}{147}s},
    \end{align*}
    where we used the fact that \(|\eta^{ex}(z,s)/\mathcal{P}_{\theta}(z,s)| \ll 1\). This concludes the exterior nonlinear estimate.
\end{proof}

\begin{remark}
\label{rmk:veph^2 estimate}
    The case when \(k=2\) in the proof of (iii) in Lemma \ref{Lemma4.1} shows that \(\|\veph^2\|_{L^2_{\rho}(\mathbb{R}^2)}=o(se^{-3s})\).
\end{remark}

\medskip

\subsection{Decomposition of the generated error \(E\)}
\begin{lemma}
\label{Lemma4.2}
Let \(s_0 = s_0(A_1, A_2,A_3,K,B_1,B_2,B_3,C_1,C_2,C_3,D_1) \gg 1\), let \(\varepsilon(s) \in S(s)\) be a solution for \ref{semilinearpde} for all \(s \in [s_0,s_1]\) with \(s_1<\infty\). Then the generated error \(E\) defined in \ref{E} satisfies\\

(i) \(|E(y,s)| \lesssim se^{-3s}|y|^{10}\) for \(|y| \leq 2Ke^{s/4}\).\\

(ii) \(|\mathcal{E}(z,s)| \lesssim \min{\{e^{-s/3},e^{-s/2}|z|^2\}} \) for \(|z|  \geq 2K\), where \(\mathcal{E}(z,s)=E(y,s)\).
\end{lemma}

\begin{proof}
(i) Recall from (\ref{E}) that
\begin{equation*}
    E(y,s)=-\partial_s\mathcal{P}_{\theta}+\Delta \mathcal{P}_{\theta}-\frac{1}{2}y \cdot \nabla \mathcal{P}_{\theta}+\frac{1}{3}\mathcal{P}_{\theta}-\frac{1}{\mathcal{P}_{\theta}^2}.
\end{equation*}
Note that it is suffices to show \(|E(y,s)| \lesssim se^{-3s}|y|^{10}\) in the region \(\left\{|y| \leq e^{\frac{s}{4}}\right\}\), where we have 
\begin{align*}
   \forall |y| \leq e^{s/4}, \quad  \mathcal{P}_{\theta}(y,s)=\left(3+e^{-s}y_1^2y_2^2+\theta e^{-2s}(y_1^6+y_2^6)\right)^{\frac{1}{3}}+\mathcal{C}(y,s), \quad  \theta \in (0,\theta^*),
\end{align*}
and \(\mathcal{C}(y,s)\) is the correction term defined in \ref{corr}.\\

Calculating each term of \(E\) yields:
\begin{align*}
    E(y,s)&=-\partial_s\mathcal{C}_1+\Delta \mathcal{C}_1-\frac{1}{2}y \cdot \nabla \mathcal{C}_1+\mathcal{C}_1\\
    & \quad + \frac{10}{\hbar} \delta e^{-2s}(y_1^4+y_2^4))-\frac{8}{9\hbar^5}e^{-2s}(y_1^4y_2^2+y_1^2y_2^4)\\
    & \quad +\left(\frac{8}{81\hbar}e^{-2s}y_1^4y_2^2+\frac{8}{81\hbar}e^{-2s}y_1^2y_2^4-\frac{16}{9\hbar^2}e^{-2s}y_1^2y_2^2\right)\\
    & \quad + \sqrt{2}e^{-2s}\left(\frac{4}{81\hbar}y_1^4+\frac{4}{81\hbar}y_2^4+\frac{16}{\hbar^3}+\frac{8}{81\hbar}y_1^2y_2^2-\frac{16}{9\hbar^2}y_1^2-\frac{16}{9\hbar^2}y_2^2\right)\\
    & \quad +se^{-3s}\langle y\rangle^{10},
\end{align*}
for \(|y| \leq e^{s/4}\), where 
\begin{align*}
\mathcal{C}_1(y,s)=\mathcal{C}(y,s)-\frac{\hbar}{9} e^{-s} (-2 y_1^2 - 2 y_2^2 + 4), \quad \hbar=3^{\frac{1}{3}}.
\end{align*}
Combining the above terms leads to the estimate
\begin{align*}
    |E(y,s)|\lesssim se^{-3s}|y|^{10},
\end{align*}
for \(|y| \leq e^{s/4}\). This establishes the desired estimate.\\
(ii) Recall that by \ref{E}, \(\mathcal{E}(z,s)\) reads as
\begin{align*}
    \mathcal{E}(z,s)=-\partial_s\mathcal{P}_{\theta}+e^{-s/2}\Delta \mathcal{P}_{\theta}-\frac{1}{4}(z \cdot \nabla)\mathcal{P}_{\theta}+\frac{1}{3}\mathcal{P}_{\theta}-\frac{1}{\mathcal{P}_{\theta}^2},
\end{align*}
for \(|z| \geq 2K\), where \(\mathcal{P_{\theta}=\mathcal{P_{\theta}}}(z,s)\) is defined in \ref{qp}. We now calculate \(\mathcal{E}(z,s)\) in the regions\\ \(\{2K \leq |z| \leq \frac{1}{4}e^{\frac{s}{4}}\}\), \(\{\frac{1}{4}e^{\frac{s}{4}} \leq |z| \leq \frac{1}{2}e^{\frac{s}{4}}\}\) and \(\{\frac{1}{2}e^{\frac{s}{4}} \leq |z| \leq e^{\frac{s}{4}}\}\), respectively.\\

For \(\{2K \leq |z| \leq \frac{1}{4}e^{\frac{s}{4}}\}\), we have \(\mathcal{P}_{\theta}=(3+z_1^2z_2^2+\theta e^{-\frac{s}{2}}(z_1^6+z_2^6))^{\frac{1}{3}}\) and observe that
\begin{align*}
    -\partial_s\mathcal{P}_{\theta}-\frac{1}{4}(z \cdot \nabla)\mathcal{P}_{\theta}+\frac{1}{3}\mathcal{P}_{\theta}-\frac{1}{\mathcal{P}_{\theta}^2}=0,
\end{align*}
thus, it is left to estimate \(e^{-s/2} \Delta \mathcal{P}_{\theta}\). By a direct computation, we have the estimate
\begin{align*}
    |e^{-s/2}\Delta \mathcal{P}_{\theta}|& \lesssim \frac{e^{-s/2}|z|^2}{(3+z_1^2z_2^2+\theta e^{-s/2}(z_1^6+z_2^6))^{5/3}}+\frac{e^{-s/2}z_1^2z_2^2|z|^2}{(3+z_1^2z_2^2+\theta e^{-s/2}(z_1^6+z_2^6))^{5/3}}\\
    &\quad +\frac{e^ {-s}|z|^8}{(3+z_1^2z_2^2+\theta e^{-s/2}(z_1^6+z_2^6))^{5/3}}\\
    & \lesssim \min{\{e^{-s/3},e^{-s/2}|z|^2\}}.
\end{align*}
For \(\{\frac{1}{4}e^{\frac{s}{4}} \leq |z| \leq \frac{1}{2}e^{\frac{s}{4}}\}\), we have to take account for the cutoff function \(\mchi_{\kappa(s)}\), where we have
\begin{align*}
    \mathcal{P}_{\theta}(y,s) = \Psi_{\theta}(z,s)\mchi_{\frac{1}{4}\kappa(s)}(z,s)+e^{\frac{s}{3}}\left(1-\mchi_{\frac{1}{4}\kappa(s)}(z,s)\right).
\end{align*}
By a direct calculation, we derive
\begin{align*}
    |\mathcal{E}(z,s)| \lesssim \min{\{e^{-s/3},e^{-s/2}|z|^2\}}.
\end{align*}
Finally, for \(\frac{1}{2}e^{\frac{s}{4}} \leq |z| \leq e^{\frac{s}{4}}\), we have \(\mathcal{P}_{\theta}(z,s)=e^{\frac{s}{3}}\), for \(\frac{1}{2}e^{\frac{s}{4}} \leq |z| \leq e^{\frac{s}{4}}\), which implies
\begin{align*}
    \mathcal{E}(z,s)=e^{-\frac{2}{3}s}.
\end{align*}
Combining the above estimates yields
\begin{align*}
    |\mathcal{E}(z,s)| \lesssim \min{\{e^{-s/3},e^{-s/2}|z|^2\}}, \quad |z| \geq 2K,
\end{align*}
which concludes the proof of Lemma \ref{Lemma4.2}.
\end{proof}

\vspace{10pt}

\section{Control of the Solution Within the Bootstrap Regime}
\label{section5}
We prove Proposition \ref{Proposition3.2} to complete the proof of Theorem \ref{Theorem3.1} (hence, Theorem \ref{Theorem 1.1}). Throughout this section, we fix \(s_1 \gg s_0\) and assume \(\varepsilon(s) \in S(s)\) for all \(s \in [s_0,s_1]\). We shall use energy estimates and the comparison principle to show that the bootstrap estimates \ref{in2}, \ref{in3}, \ref{in4}, \ref{in5}, \ref{int1}, \ref{bootstrap: int 2} and \ref{out2} hold for all \(s \in [s_0,s_1]\) with a strict inequality, which is the conclusion of Proposition \ref{Proposition3.2}.

\subsection{Control of \(\veph\) in \text{Compact Sets}}
\label{Control of veph in K}

We improve the bootstrap constants \(A_1\), \(A_2\) and \(A_3\) by establishing energy estimates for \(\|\veph_-\|_{L^2_{\rho}(\mathbb{R}^2)}\), \(\|\partial_i \veph_-\|_{L^2_{\rho}(\mathbb{R}^2)}\), \(\|\partial^2 _{ij}\veph_-\|_{L^2_{\rho}(\mathbb{R}^2)}\), where \(\veph=\veph(y,s)\), and establishing an ODE for \(|\veph_{ij}(s)|\) for \(i+j=8\), thereby showing that the bootstrap conditions \ref{in2}, \ref{in3}, \ref{in4}, \ref{in5} hold for all \(s \in [s_0,s_1]\) with strict inequalities.\\

Note that by definition, we have 
\(\veph(y,s)=\varepsilon(y,s) \mchi_{K}(z)\). Combining this with the PDE (\ref{semilinearpde}), we have
\begin{align}
    \partial_s\veph=\mathcal{L}\veph+\hat{V}(y,s)\veph+NL(\veph)+\hat{E}+\hat{R},\label{innerpde}
\end{align}
where \(\hat{E}:=E(y,s)\mchi_K(z)\), and by Lemma \ref{Lemma4.2}, we can derive \(\|\hat{E}\|_{L^2_{\rho}(\mathbb{R}^2)}=o(se^{-3s})\); the potential term \(\hat{V}(y,s)=\frac{2}{P_{\theta}^3}-\frac{2}{3}\) and satisfies the estimate \(|\hat{V}(y,s)| \lesssim e^{-s}|y|^4\) for \(0 \leq |y| \leq 2Ke^{s/4}\); \(NL(\veph)\) is the nonlinear term defined in Lemma \ref{Lemma4.1} (iii) and satisfies \(\|NL(\veph)\|_{L^2_{\rho}(\mathbb{R}^2)}=o(se^{-3s})\). Finally, 
\(\hat{R}\) is defined as 
\begin{align}
    \hat{R}&=\left(\frac{1}{2}y \cdot \nabla \mchi_K(z)\right) \varepsilon-2 \nabla \varepsilon \cdot \nabla \mchi_K(z) -\varepsilon \Delta \mchi_K(z) \nonumber \\
    & \qquad \qquad + \sum_{k=2}^{\infty}\binom{-2}{k}\left(\frac{\varepsilon^k\mchi_K(z)(1-\mchi_K^{k-1}(z))}{\mathcal{P}_{\theta}^{k+2}}\right)-\varepsilon \cdot \partial_s\mchi_K(z), \label{eq:hatR}
\end{align}
where the derivatives are taken with respect to the variable \(y\).
\begin{lemma}
\label{lem:estimR}
    We have the estimate
    \begin{align*}
        \|\hat{R}\|_{L^2_{\rho}(\mathbb{R}^2)}=o(se^{-3s}).
    \end{align*}
\end{lemma}
\begin{proof}
    Note that by Lemma \ref{Lemma4.1} (ii), we have \(|\varepsilon(y,s)| \lesssim e^{-\frac{5}{12}s}\) in \(Ke^{s/4} \leq |y| \leq 2Ke^{s/4}\). Moreover, by Remark \ref{rmk: L2 rho control outer}, we have \(|\nabla\varepsilon(y,s)| \lesssim e^{-\frac{8}{12}s}\) for \(Ke^{s/4} \leq |y
    |\leq 2Ke^{s/4}\). Thus,  
    \(|\hat{R}(y,s)| \lesssim e^{-\frac{5}{12}s}\) for all \(y \in \mathbb{R}^2\) and hence
    \begin{align*}
        \|\hat{R}\|_{L^2_{\rho}(\mathbb{R}^2)}&=\left(\int_{\mathbb{R}^2}|\hat{R}(y,s)|^2\rho(y)\,dy\right)^{\frac{1}{2}} \lesssim e^{-\frac{5}{12}s}\left( \int_{|y| \geq e^{s/4}}\rho(y)\,dy\right)^{\frac{1}{2}} = o(se^{-3s}).
    \end{align*}
\end{proof}

\begin{remark}
\label{rmk:partial ij varepslion}
    By Remark \ref{rmk: L2 rho control outer} and Remark \ref{rmk: L2 rho control inner intermd control on boundary}, we also have 
    \begin{align*}
             |\partial^2_{ij}\varepsilon(y,s)| \lesssim e^{-\frac{11}{12}s}, \quad
              |\partial^3_{ijk}\varepsilon(y,s)| \lesssim e^{-\frac{14}{12}s}, \quad i,j,k=1,2,
         \end{align*}
         in \(Ke^{s/4} \leq |y| \leq 2Ke^{s/4}\).
\end{remark}
\begin{lemma}
\label{lem:estm Vveph}
\label{lem:Vveph}
    We have the estimate 
    \begin{align*}
        \|\hat{V}\veph(s)\|_{L^2_{\rho}(\mathbb{R}^2)}=o(se^{-3s}).
    \end{align*}
\end{lemma}
\begin{proof}
    By using the Cauchy-Schwarz inequality and \(|\hat{V}(y,s)| \lesssim e^{-s}|y|^4\) for \(0 \leq |y| \leq 2Ke^{s/4}\), we have 
    \begin{align*}
        \|\hat{V}\veph(s)\|^2_{L^2_{\rho}(\mathbb{R}^2)}&=\int_{\mathbb{R}^2}|\hat{V}(y,s)|^2|\veph(y,s)|^2\rho(y)\,dy\\
        & \lesssim e^{-2s}\left(\int_{\mathbb{R}^2}|y|^{16}\rho(y)\,dy\right)^{1/2}\left(\int_{\mathbb{R}^2}|\veph(y,s)|^4\rho(y)\,dy\right)^{1/2} \lesssim e^{-\frac{44}{6}s},
    \end{align*}
where we used
    \begin{align*}
    \left(\int_{\mathbb{R}^2}|\veph(y,s)|^4\rho(y)\,dy\right)^{1/2} \lesssim e^{-\frac{32}{6}s},
    \end{align*}
from Remark \ref{rmk:veph^2 estimate}. This shows that \(\|\hat{V}\veph(s)\|_{L^2_{\rho}(\mathbb{R}^2)}=o(se^{-3s})\).
\end{proof}
Next, by the nonlinear estimate for \(\veph\) in Lemma \ref{Lemma4.1} (iii), Lemma \ref{lem:estimR}, and Lemma \ref{lem:Vveph}, we can further write the PDE (\ref{innerpde}) as
\begin{align}
    \partial_s\veph=\mathcal{L}\veph +\hat{G},\label{innerpdered}
\end{align}
for all \(s \in [s_0,s_1]\), where \(\hat{G}=\hat{V}(y,s)\veph+NL(\veph)+\hat{E}+\hat{R}\) and satisfies \(\|\hat{G}\|_{L^2_{\rho}(\mathbb{R}^2)}=o(se^{-3s})\).

With the spectral gap \ref{inq:spec gap}, we prove the monotonicity of $\|\varepsilon(s)\|_{L^2_\rho}$ in the following.
\begin{lemma}
\label{Lemma 5.2}
    We have the energy estimate
    \begin{align}
        \frac{1}{2}\frac{d}{ds}\|\veph_-\|^2_{L^2_{\rho}(\mathbb{R}^2)} \leq -\frac{7}{2}\|\veph_-\|^2_{L^2_{\rho}(\mathbb{R}^2)}+A_1 e^{-\frac{32}{6}s},\label{en1}
  \end{align}
  \end{lemma}

\begin{proof}
    By equation (\ref{innerpdered}), bootstrap condition (\ref{in3}), the \(L^{2}_{\rho}\) estimate of \(\hat{G}\) above, and the spectral gap \ref{inq:spec gap}, we have
    \begin{align*}
        \frac{1}{2}\frac{d}{ds}\|\veph_-\|^2_{L^2_{\rho}(\mathbb{R}^2)} &\leq \left<\mathcal{L}\veph_-,\veph_-\right>_{\rho}+\left<\hat{G},\veph_-\right>_{\rho} \leq -\frac{7}{2}\|\veph_-\|^2_{L^2_{\rho}(\mathbb{R}^2)}+A_1e^{-\frac{32}{6}s},
    \end{align*}
which is the desired energy estimate.
\end{proof}

\begin{lemma}
\label{Lemma 5.3}
We have 
    \begin{align}
       i + j = 8, \quad  |\veph_{ij}'(s)+3\veph_{ij}(s)| \leq \frac{A_1}{K}e^{-\frac{32}{12}s}.\label{en2}
    \end{align}
\end{lemma}
\begin{proof}
    Recall that we have the decomposition \(\veph(y,s)=\sum_{i+j \leq 8} \veph_{ij}(s)H_{ij}(y)\). Using this expansion and projecting equation (\ref{innerpdered}) onto each eigenfunction \(H_{ij}\) yields
    \begin{align}
        \veph'_{ij}(s)=\left(1-\frac{i+j}{2}\right)\veph_{ij}(s)+\frac{1}{\|H_{ij}\|_{L^2_{\rho}(\mathbb{R}^2)}}\left<\hat{G},H_{ij}\right>_{\rho}.\label{pdecom}
    \end{align}
    Since \(i+j=8\), and 
    \begin{align*}
        \frac{1}{\|H_{ij}\|_{L^2_{\rho}(\mathbb{R}^2)}}\left|\left<\hat{G},H_{ij}\right>_{\rho} \right| \leq \frac{A_1}{K}e^{-\frac{32}{12}s},
    \end{align*}
    we derive the estimate
    \begin{align*}
        |\veph_{ij}'(s)+3\veph_{ij}(s)| \leq \frac{A_1}{K}e^{-\frac{32}{12}s},
    \end{align*}
    which concludes the Lemma \ref{Lemma 5.3}.
\end{proof}

We now proceed to establish energy estimates for
\(\|\partial_i\veph_-\|_{L^2_{\rho}(\mathbb{R}^2)}\) and \(\|\partial^2_{ij}\veph_-\|_{L^2_{\rho}(\mathbb{R}^2)}\) for \(i,j=1,2\).\\

Note that by Remark \ref{rmk: L2 rho control inner intermd control on boundary} in Appendix \ref{appdx}, we have
\begin{align}
    &\|\partial_i\veph(s)\|_{L^{\infty}(|y| \leq 50K)} \lesssim C(K) e^{-\frac{32}{12}s},\label{inq:ptwise local i}\\
    &\|\partial^2_{ij}\veph(s)\|_{L^{\infty}(|y| \leq 50K)} \lesssim C(K)e^{-\frac{32}{12}s},\label{inq:ptwise local ij}
\end{align}
and by remark \ref{rmk: L2 rho control outer} in Appendix \ref{appdx}, we derive the pointwise estimates for \(4K \leq |y| \leq 2Ke^{\frac{s}{4}}\):
\begin{align}
    &|\partial_i\veph(y,s)| \lesssim  e^{-\frac{32}{12}s}|y|^8,\label{inq:ptwise i geq K}\\
    & |\partial^2_{ij}\veph(y,s)|  \lesssim  e^{-\frac{32}{12}s}|y|^7.\label{inq:ptwise ij geq K}
\end{align}

\begin{lemma}
\label{lem:energy par i veph}
    We have the energy estimate
    \begin{align}
        \frac{1}{2}\partial_s\|\partial_i\veph_-\|^2_{L^2_{\rho}(\mathbb{R}^2)} \leq -\frac{7}{2}\|\partial_i\veph_-\|^2_{L^2_{\rho}(\mathbb{R}^2)}+A_2e^{-\frac{32}{6}s}, \quad i,j=1,2.\label{en6}
    \end{align}
\end{lemma}
\begin{proof}
    Differentiating equation \ref{innerpde} yields
    \begin{align}
    \label{eq:par veph i}
        \partial_s(\partial_i\veph)=\mathcal{L}(\partial_i\veph)-\frac{1}{2}\partial_i \veph+\partial_i\hat{V}\,\veph+\hat{V}\partial_i\,\veph+\partial_i \hat{E}+\partial_iNL(\veph)+\partial_i\hat{R}.
    \end{align}
    Using the pointwise estimates \ref{inq:ptwise local i} and \ref{inq:ptwise i geq K}, Remark \ref{rmk:partial ij varepslion}, Lemma \ref{lem:estm Vveph}, the nonlinear estimate for \(\veph\) in Lemma \ref{Lemma4.1} (iii) and an analysis similar to Lemma \ref{lem:estimR}, we have
    \begin{align*}
    \partial_i\hat{V}\,\veph+\hat{V}\partial_i\,\veph+\partial_i \hat{E}+\partial_iNL(\veph)+\partial_i\hat{R}=o(se^{-3s}),
    \end{align*}
    in \(L^2_{\rho}(\mathbb{R}^2)\). Thus, using the orthogonal condition \ref{con:orth 1} and by projecting equation \ref{eq:par veph i} with \(\partial_i\veph_-\) in \(L^2_{\rho}(\mathbb{R}^2)\), combined with the spectral gap \ref{inq:spec partial i} yields
    \begin{align*}
        \frac{1}{2}\partial_s\|\partial_i\veph_-\|^2_{L^2_{\rho}(\mathbb{R}^2)} 
        & \leq -\frac{7}{2}\|\partial_i\veph_-\|^2_{L^2_{\rho}(\mathbb{R}^2)}+A_2e^{-\frac{32}{6}s}.
    \end{align*}
    This concludes the energy estimate for \(\|\partial_i\veph\|_{L^2_{\rho}(\mathbb{R}^2)}\).
\end{proof}

By using the orthogonal condition \ref{con:orth 2} and the pointwise estimate \ref{inq:ptwise local ij}, \ref{inq:ptwise ij geq K}, combined with the pointwise estimates in Remark \ref{rmk:partial ij varepslion} and the spectral gap \ref{inq:spec partial ij}, we may prove the following Lemma by a similar fashion to Lemma \ref{lem:energy par i veph}, concluding the energy estimate for \(\|\partial^2_{ij}\veph_-\|_{L^2_{\rho}(\mathbb{R}^2)}\).
\begin{lemma}
\label{lem:energy par ij veph}
    We have the energy estimate
    \begin{align}
        \frac{1}{2}\partial_s\|\partial_{ij}\veph_-\|^2_{L^2_{\rho}(\mathbb{R}^2)} \leq -\frac{7}{2}\|\partial_{ij}\veph_-\|^2_{L^2_{\rho}(\mathbb{R}^2)}+A_3e^{-\frac{32}{6}s}, \quad i,j=1,2.\label{en7}
    \end{align}
\end{lemma}

We end this subsection by proving (\ref{strictly}) in Proposition \ref{Proposition3.3}, which is the following Lemma.

\begin{lemma}
\label{Lemma5.4}
    We have the strictly outgoing property 
    \begin{align*}
        \frac{1}{2}\frac{d}{ds}\left(\sum_{i+j \leq 7}\left|\frac{e^{\frac{32}{12}\bar s}}{A_1}\veph_{ij}(\bar{s})\right|^2\right)>0,
    \end{align*}
    provided that \(s_0 \gg 1\), where \(\bar{s}\) is the exit time introduced in the proof of Proposition \ref{Proposition3.3}.
\end{lemma}

\begin{proof}
    Note that by applying the ODE (\ref{pdecom}), we have 
    \begin{align*}
        \frac{d}{ds}\frac{1}{2}\sum_{i+j \leq 7}\left|\frac{e^{\frac{32}{12}s}}{A_1}\veph_{ij}(s)\right|^2&=\sum_{i+j \leq 7}\frac{e^{\frac{32}{12}s}}{A_1}\veph_{ij}(s)\left(\frac{32}{12}\frac{e^{\frac{32}{12}s}}{A_1}\veph_{ij}(s)+\frac{e^{\frac{32}{12}s}}{A_1}\veph'_{ij}(s)\right)\\
        & \geq \sum_{i+j \leq 7}\frac{e^{\frac{32}{6}s}}{A_1^2}|\veph_{ij
        }(s)|^2\left(\frac{32}{12}-\frac{5}{2}+\frac{1}{\|H_{ij}\|_{L^2_{\rho}(\mathbb{R}^2)}}\left<\hat{G},H_{ij}\right>_{\rho}\right),
    \end{align*}
    therefore,
    \begin{align*}
        \frac{d}{ds}\frac{1}{2}\sum_{i+j \leq 7}\left|\frac{e^{\frac{32}{12}\bar s}}{A_1}\veph_{ij}(\bar{s})\right|^2 &\geq \frac{2}{12}+\frac{1}{\|H_{ij}\|_{L^2_{\rho}(\mathbb{R}^2)}}\left<\hat{G},H_{ij}\right>_{\rho} \geq \frac{2}{12}+O(\bar{s}e^{-3\bar{s}})>0.
     \end{align*}
    provided that \(s_0 \gg 1\). This shows the strictly outgoing property (\ref{strictly}).
\end{proof}

\subsection{Control of \(\|\varepsilon\|_{\flat}\)}
\label{subsection5.2}
In this subsection, we estimate the gradient terms \((y \cdot \nabla)^k \varepsilon\), for \(k = 0, 1, 2\), in appropriate weighted \(H^2\) spaces by deriving energy estimates for \(\|(y \cdot \nabla)^k \varepsilon\|_{\flat}\). These estimates allow us to refine the bootstrap constants \(B_1\), \(B_2\), and \(B_3\), thereby showing that the bootstrap condition \ref{int1} is strictly satisfied for all \(s \in [s_0, s_1]\). Recall that we have the PDE (see \ref{semilinearpde})
\begin{align}
    \partial_s \varepsilon=\Delta \varepsilon-\frac{1}{2}y \cdot \nabla \varepsilon+\frac{1}{3}\varepsilon+NL(\varepsilon)+E,
\end{align}
where \(E\) is defined in \ref{E} and \(NL(\varepsilon)\) is defined in \ref{NL}.
\begin{remark}
\label{inq:bdd term estimate}
In particular, by \ref{inq:ptwise local i} and \ref{inq:ptwise local ij}, we have the pointwise estimates at \(|y|=K\):
\begin{align*}
& |\varepsilon(y,s)|\lesssim C(K)e^{-\frac{32}{12}s}, \quad |\partial_i\varepsilon(y,s)| \lesssim C(K)e^{-\frac{32}{12}s}, \quad
|\partial_{ij}\varepsilon(y,s)| \lesssim C(K)e^{-\frac{32}{12}s},\quad i,j=1,2,\\
&|(y \cdot \nabla)\varepsilon(y,s)| \lesssim C(K)e^{-\frac{32}{12}s},\quad |\nabla(y \cdot \nabla)\varepsilon(y,s)| \lesssim C(K)e^{-\frac{32}{12}s},\quad |(y \cdot \nabla)^2 \varepsilon(y,s)| \lesssim C(K)e^{-\frac{32}{12}s}.
\end{align*}
Moreover, by Remark \ref{rmk: L2 rho control inner intermd control on boundary} in Appendix \ref{appdx}, we also have the following estimate at \(|y|=K\):
\begin{align}
    &|\nabla(y \cdot \nabla)^2 \varepsilon(y,s)| \lesssim C(K)e^{-\frac{32}{12}s}.\label{inq:K pt nabla(y nabla)^2 veph}
\end{align}
\end{remark}
\begin{remark}
    Let \(\nu\) be the outward pointing unit normal vector field of \(\{|y|=K\}\).
\end{remark}
   
\begin{lemma}
\label{Lemma5.5}
    We have the energy estimate
    \begin{align}
        \frac{1}{2}\frac{d}{ds}\|\varepsilon\|_{\flat}^2 \leq-\frac{7}{2}\|\varepsilon\|^2_{\flat}+\frac{B_1^2}{K} e^{-\frac{32}{6}s},\label{en3}
    \end{align}
   where \(\|\cdot\|_{\flat}\) is the norm (\ref{int2}).
\end{lemma}

\begin{proof}
We compute
\begin{align*}
    \frac{1}{2}\frac{d}{ds}\|\varepsilon\|_{\flat}^2 &= \int_{|y| \geq K} \frac{\varepsilon\, \partial_s \varepsilon}{|y|^{\alpha}}\, \frac{dy}{|y|^2}\\
    &= \int_{|y| \geq K} \frac{\varepsilon\, \Delta \varepsilon}{|y|^{\alpha}}\, \frac{dy}{|y|^2} 
    - \frac{1}{2} \int_{|y| \geq K} \frac{\varepsilon\, (y \cdot \nabla \varepsilon)}{|y|^{\alpha}}\, \frac{dy}{|y|^2} 
    +\frac{1}{3} \int_{|y| \geq K} \frac{|\varepsilon|^2}{|y|^{\alpha}}\, \frac{dy}{|y|^2}\\
    & \quad+ \int_{|y| \geq K} \frac{\varepsilon \cdot NL(\varepsilon)}{|y|^{\alpha}}\, \frac{dy}{|y|^2}+\int_{|y| \geq K}\frac{ \varepsilon \cdot E}{|y|^{\alpha}}\,\frac{dy}{|y|^2}.
\end{align*}
For the Laplacian term, using integration by parts, we have
\begin{align*}
    \int_{|y| \geq K} \frac{\varepsilon \Delta \varepsilon}{|y|^{\alpha}}\, \frac{dy}{|y|^2} 
    &=-\frac{1}{2}(\alpha+2)^2\int_{|y| \geq K}\frac{|\varepsilon|^2}{|y|^{\alpha+2}}\frac{dy}{|y|^2}+\frac{1}{2}\int_{|y|=K}\left(\frac{1}{|y|^{\alpha+2}}\frac{\partial(\varepsilon^2)}{\partial \nu}-\varepsilon^2\frac{\partial(|y|^{-\alpha-2})}{\partial \nu}\right)\,d\sigma\\
    & \quad -\int_{|y| \geq K}\frac{|\nabla \varepsilon|^2}{|y|^{\alpha}}\frac{dy}{|y|^2} \leq B_1e^{-\frac{32}{6}s}
\end{align*}
where the boundary terms are controlled by Remark \ref{inq:bdd term estimate}. For the second term, by integration by parts, we have
\begin{align*}
    -\frac{1}{2} \int_{|y| \geq K} \frac{\varepsilon (y \cdot \nabla \varepsilon)}{|y|^{\alpha}}\,  \frac{dy}{|y|^2} 
    &= \left( -\frac{\alpha}{4} \right) \int_{|y| \geq K} \frac{\varepsilon^2}{|y|^{\alpha}}\,  \frac{dy}{|y|^2}+B_1 e^{-\frac{32}{6}s},
\end{align*}
where the boundary terms are estimated by Remark \ref{inq:bdd term estimate}. Next, using the estimates (ii) and (iv) in Lemma \ref{Lemma4.1} , we have
\begin{align*}
    \int_{|y| \geq K} \frac{\varepsilon \cdot E}{|y|^{\alpha}}\frac{dy}{|y|^2}+\int_{|y| \geq K} \frac{\varepsilon \cdot NL(\varepsilon)}{|y|^{\alpha}}\frac{dy}{|y|^2}&=\int_{K \leq |y| \leq e^{\frac{9}{33}s}} \frac{\varepsilon \cdot E}{|y|^{\alpha}}\frac{dy}{|y|^2}+\int_{K \leq |y| \leq e^{\frac{9}{33}s}} \frac{\varepsilon \cdot NL(\varepsilon)}{|y|^{\alpha}}\frac{dy}{|y|^2}\\
    & \quad +\int_{e^{\frac{1}{44}s} \leq |z| \leq e^{\frac{5}{109}s}} \frac{\varepsilon \cdot E}{|y|^{\alpha}}\frac{dy}{|y|^2}+\int_{e^{\frac{1}{44}s} \leq |z| \leq e^{\frac{5}{109}s}} \frac{\varepsilon \cdot NL(\varepsilon)}{|y|^{\alpha}}\frac{dy}{|y|^2}\\
    & \quad +\int_{e^{\frac{5}{109}s} \leq |z| \leq e^{\frac{5}{49}s}} \frac{\varepsilon \cdot E}{|y|^{\alpha}}\frac{dy}{|y|^2}+\int_{e^{\frac{5}{109}s} \leq |z| \leq e^{\frac{5}{49}s}} \frac{\varepsilon \cdot NL(\varepsilon)}{|y|^{\alpha}}\frac{dy}{|y|^2}\\
    & \quad +\int_{e^{\frac{5}{49}s} \leq |z| \leq e^{\frac{1}{4}s}} \frac{\varepsilon \cdot E}{|y|^{\alpha}}\frac{dy}{|y|^2}+\int_{e^{\frac{5}{49}s} \leq |z| \leq e^{\frac{1}{4}s}} \frac{\varepsilon \cdot NL(\varepsilon)}{|y|^{\alpha}}\frac{dy}{|y|^2}\\
    & \leq \frac{2}{P_{\theta}^3} \int_{|y|\geq K}\frac{|\varepsilon|^2}{|y|^{\alpha}}\frac{dy}{|y|^2}+o(e^{-\frac{32}{6}s}).
\end{align*}
Here, the technical constant \(\frac{9}{33}\) is chosen so that \(|E(y,s)| \lesssim se^{-3s}|y|^{10} \lesssim e^{-\frac{32}{12}s}|y|^{8.9}\) throughout \(0 \leq |y| \leq e^{\frac{9}{33}s}\) (equivalently, for \(0 \leq |z| \leq e^{\frac{1}{44}s}\)). Next, the technical constant \(\frac{5}{109}\) is chosen to simplify the calculations of the nonlinear estimate for \(0 \leq |z| \leq e^{\frac{5}{109}s}\) (equivalently, for \(0 \leq |y| \leq e^{(\frac{1}{4}+\frac{5}{109})s}\)). Finally, the technical constant \(\frac{5}{49}\) is chosen to simplify the nonlinear estimate for \(e^{\frac{5}{109}s}\leq |z| \leq e^{\frac{5}{49}s}\) as well as to control the nonlinear estimate for \(|z| \geq e^{\frac{5}{49}s}\). Combining the above estimates and using \(\alpha=18\) yields
\begin{align}
        \frac{1}{2}\frac{d}{ds}\|\varepsilon\|_{\flat}^2 \leq-\frac{7}{2}\|\varepsilon\|^2_{\flat}+\frac{B_1^2}{K} e^{-\frac{32}{6}s},
\end{align}
which concludes the desired energy estimate.
\end{proof}
\begin{lemma}
\label{Lemma5.6}
    We have the energy estimate
    \begin{align}
        \frac{1}{2}\frac{d}{ds}\|(y \cdot \nabla \varepsilon)\|_{\flat}^2 \leq-\frac{7}{2}\|(y \cdot \nabla \varepsilon)\|^2_{\flat}+\frac{B_2^2}{K} e^{-\frac{32}{6}s}.\label{en4}
    \end{align}
   
\end{lemma}
\begin{proof}
We compute:
\begin{align*}
    \frac{1}{2}\frac{d}{ds}\|(y \cdot \nabla)\varepsilon\|_{\flat}^2 
    &= \frac{1}{2}\frac{d}{ds}\int_{|y| \geq K}\frac{|(y \cdot \nabla)\varepsilon|^2}{|y|^{\alpha}}\,\frac{dy}{|y|^2} \\
    &= \int_{|y| \geq K}\frac{(y \cdot \nabla)\varepsilon \cdot (y \cdot \nabla)(\Delta \varepsilon)}{|y|^{\alpha}}\,\frac{dy}{|y|^2}
    - \frac{1}{2} \int_{|y| \geq K}\frac{(y \cdot \nabla)\varepsilon \cdot (y \cdot \nabla)^2\varepsilon}{|y|^{\alpha}}\,\frac{dy}{|y|^2} \\
    &\quad +\frac{1}{3} \int_{|y| \geq K}\frac{|(y \cdot \nabla)\varepsilon|^2}{|y|^{\alpha}}\,\frac{dy}{|y|^2} 
    + \int_{|y| \geq K}\frac{(y \cdot \nabla)\varepsilon \cdot (y \cdot \nabla)(NL(\varepsilon))}{|y|^{\alpha}}\,\frac{dy}{|y|^2}\\
    &\quad +\int_{|y| \geq K}\frac{(y \cdot \nabla) \varepsilon \cdot (y \cdot \nabla)E}{|y|^{\alpha}}\,\frac{dy}{|y|^2}.     
\end{align*}
We now estimate each term in this energy identity. For the Laplacian term, integrating by parts yield
\begin{align*}
    &\int_{|y| \geq K}\frac{(y \cdot \nabla)\varepsilon \cdot (y \cdot \nabla)(\Delta \varepsilon)}{|y|^{\alpha}}\,\frac{dy}{|y|^2}\\
    & \quad = \int_{|y| \geq K}\frac{(y \cdot \nabla) \varepsilon \cdot \Delta(y \cdot \nabla)
    \varepsilon}{|y|^{\alpha}}\,\frac{dy}{|y|^2}-2\int_{y \geq K}\frac{(y \cdot \nabla)\varepsilon \: \Delta\varepsilon}{|y|^{\alpha}}\,\frac{dy}{|y|^2}\\
    & \quad \leq -\frac{1}{2}(\alpha+2)^2 \int_{|y| \geq K}\frac{|y \cdot \nabla\varepsilon|^2}{|y|^{\alpha+2}}\,dy-\int_{|y| \geq K}\frac{|\nabla(y \cdot \nabla \varepsilon)|^2}{|y|^{\alpha+2}}\,dy-2\int_{|y| \geq K}\frac{(y \cdot \nabla \varepsilon) \Delta \varepsilon}{|y|^{\alpha+2}}\,dy + B_2e^{-\frac{32}{6}s} \\
    & \quad \leq 2B_2e^{-\frac{32}{6}s},
\end{align*}
where the boundary terms are controlled by Remark \ref{inq:bdd term estimate} and the Laplacian term in the last line is dominated by the dissipative term. For the second term, we have, by integration by parts
\begin{align*}
    -\frac{1}{2} \int_{|y| \geq K}\frac{(y \cdot \nabla)\varepsilon \cdot (y \cdot \nabla)((y \cdot \nabla)\varepsilon)}{|y|^{\alpha}}\,\frac{dy}{|y|^2} 
    &\leq \left( -\frac{\alpha}{4} \right) \int_{|y| \geq K}\frac{|(y \cdot \nabla)\varepsilon|^2}{|y|^{\alpha}}\,\frac{dy}{|y|^2}+B_2e^{-\frac{32}{6}s},
\end{align*}
where the bounds of the boundary terms are controlled by Remark \ref{inq:bdd term estimate}. Next, by parabolic regularity (See Theorem \ref{thm:reg main} and the following remarks), we have 
\begin{align*}
     \int_{|y| \geq K} \frac{(y \cdot \nabla \varepsilon) \cdot (y \cdot \nabla E)}{|y|^{\alpha}}\frac{dy}{|y|^2}&+\int_{|y| \geq K} \frac{(y \cdot \nabla \varepsilon) \cdot (y \cdot \nabla NL(\varepsilon))}{|y|^{\alpha}}\frac{dy}{|y|^2} \\
     & \qquad \leq \frac{2}{P_{\theta}^3} \int_{|y|\geq K}\frac{|(y \cdot \nabla \varepsilon)|^2}{|y|^{\alpha}}\frac{dy}{|y|^2}+o(e^{-\frac{32}{6}s}),
\end{align*}
where the boundary terms are estimated by Remark \ref{inq:bdd term estimate}. Combining the above estimates and using \(\alpha=18\) yields
\begin{align}
        \frac{1}{2}\frac{d}{ds}\|(y \cdot \nabla \varepsilon)\|_{\flat}^2 \leq-\frac{7}{2}\|(y \cdot \nabla \varepsilon)\|^2_{\flat}+\frac{B_2^2}{K} e^{-\frac{32}{6}s},
\end{align}
which is the desired energy estimate.
\end{proof}

\begin{lemma}
\label{Lemma5.7}
We have the energy estimate
\begin{align}
    \frac{1}{2} \frac{d}{ds} \|(y \cdot \nabla)^2 \varepsilon\|_{\flat}^2 
    \leq -\frac{7}{2} \|(y \cdot \nabla)^2 \varepsilon\|_{\flat}^2 + \frac{B_3^2}{K}e^{-\frac{32}{6}s}.
    \label{en5} 
\end{align}
\end{lemma}

\begin{proof}
We compute
\begin{align*}
    \frac{1}{2} \frac{d}{ds} \|(y \cdot \nabla)^2 \varepsilon\|_{\flat}^2 
    &= \int_{|y| \geq K} \frac{(y \cdot \nabla)^2 \varepsilon \cdot (y \cdot \nabla)^2(\Delta \varepsilon)}{|y|^{\alpha}}\, \frac{dy}{|y|^2}
    - \frac{1}{2} \int_{|y| \geq K} \frac{(y \cdot \nabla)^2 \varepsilon \cdot (y \cdot \nabla)((y \cdot \nabla)^2\varepsilon)}{|y|^{\alpha}}\,\frac{dy}{|y|^2} \\
    &\quad + \frac{1}{3} \int_{|y| \geq K} \frac{|(y \cdot \nabla)^2\varepsilon|^2}{|y|^{\alpha}}\,\frac{dy}{|y|^2}+ \int_{|y| \geq K} \frac{(y \cdot \nabla)^2 \varepsilon \cdot (y \cdot \nabla)^2 (NL(\varepsilon))}{|y|^{\alpha}}\,\frac{dy}{|y|^2}\\
    & \quad +\int_{|y| \geq K}\frac{(y \cdot \nabla)^2 \varepsilon \cdot (y \cdot \nabla)^2 E}{|y|^{\alpha}}\,\frac{dy}{|y|^2}
\end{align*}
The Laplacian term satisfies
\begin{align*}
    &\int_{|y| \geq K} \frac{(y \cdot \nabla)^2 \varepsilon \cdot (y \cdot \nabla)^2(\Delta \varepsilon)}{|y|^{\alpha}}\,\frac{dy}{|y|^2}\\
    &=\int_{|y| \geq K}\frac{(y \cdot \nabla)^2 \varepsilon \cdot \Delta((y\ \cdot \nabla)^2\varepsilon)}{|y|^{\alpha}}\,\frac{dy}{|y|^2}-4\int_{|y| \geq K}\frac{(y \cdot \nabla)^2\varepsilon \cdot \Delta(y \cdot
     \nabla \varepsilon)}{|y|^{\alpha}}\,\frac{dy}{|y|^2}-2\int_{|y| \geq K}\frac{(y \cdot \nabla)^2 \varepsilon \cdot \Delta \varepsilon}{|y|^{\alpha}}\,\frac{dy}{|y|^2}\\
     & =-\frac{1}{2}(\alpha+2)\int_{|y| \geq K}\frac{|(y \cdot \nabla)^2\varepsilon|^2}{|y|^{\alpha+4}}\,dy-\int_{|y| \geq K}\frac{|\nabla(y \cdot \nabla)^2 \varepsilon|}{|y|^{\alpha+2}}\,dy\\
     & \quad -4\int_{|y| \geq K}\frac{(y \cdot \nabla)^2\varepsilon \cdot \Delta(y \cdot
     \nabla \varepsilon)}{|y|^{\alpha}}\,\frac{dy}{|y|^2}-2\int_{|y| \geq K}\frac{(y \cdot \nabla)^2 \varepsilon \cdot \Delta \varepsilon}{|y|^{\alpha}}\,\frac{dy}{|y|^2}+B_3e^{-\frac{32}{6}s} \leq 2B_3e^{-\frac{32}{6}s}
\end{align*}
where the integral of the boundary terms is controlled by Remark \ref{inq:bdd term estimate} and the Laplacian terms in the last line is dominated by the dissipative term. Next, we have
\begin{align*}
    -\frac{1}{2} \int_{|y| \geq K} \frac{(y \cdot \nabla)^2 \varepsilon \cdot (y \cdot \nabla)((y \cdot \nabla)^2\varepsilon)}{|y|^{\alpha}}\,\frac{dy}{|y|^2} 
&= \left( -\frac{\alpha}{4} \right) \int_{|y| \geq K} \frac{|(y \cdot \nabla)^2 \varepsilon|^2}{|y|^{\alpha}}\,\frac{dy}{|y|^2}+B_3e^{-\frac{32}{12}s},
\end{align*}
where the integral of the boundary terms is estimated by Remark \ref{inq:bdd term estimate}, and by the parabolic regularity result Theorem \ref{thm:reg main}, and the following remarks, we have
\begin{align*}
     \int_{|y| \geq K} \frac{(y \cdot \nabla )^2\varepsilon \cdot (y \cdot \nabla)^2 E}{|y|^{\alpha}}\frac{dy}{|y|^2}&+\int_{|y| \geq K} \frac{(y \cdot \nabla)^2\varepsilon \cdot (y \cdot \nabla)^2 NL(\varepsilon)}{|y|^{\alpha}}\frac{dy}{|y|^2} \\
     & \quad \leq \frac{2}{P_{\theta}^3} \int_{|y|\geq K}\frac{|(y \cdot \nabla)^2 \varepsilon|^2}{|y|^{\alpha}}\frac{dy}{|y|^2}+o(e^{-\frac{32}{6}s}).
\end{align*}
Combining the above estimates, we conclude by using \(\alpha=18\)
\[
\frac{1}{2} \frac{d}{ds} \|(y \cdot \nabla)^2 \varepsilon\|_{\flat}^2 
\leq -\frac{7}{2}\|(y \cdot \nabla)^2 \varepsilon\|_{\flat}^2 + \frac{B_3^2}{K}  e^{-\frac{32}{6}s},
\]
which concludes the energy estimate.
\end{proof}
\subsection{Control of \(\|\eta\|_{\natural}\)}
We estimate the terms \((z \cdot \nabla)^k \eta\), for \(k = 0, 1, 2\), in appropriate weighted \(H^2\) spaces by deriving energy estimates for \(\|(z \cdot \nabla)^k \eta\|_{\natural}\). This improves the bootstrap constants \(C_1\), \(C_2\), and \(C_3\), thereby showing that the bootstrap condition \ref{bootstrap: int 2} is strictly satisfied for all \(s \in [s_0, s_1]\). Recall that in the \(z\) variables, \ref{selfsimpde} can be written as
\begin{align}
    \partial_s \eta=-\frac{1}{4}(z \cdot \nabla) \eta+\frac{1}{3}\eta+e^{-s/2}\Delta \eta+NL(\eta)+\mathcal{E},
\end{align}
where \(\mathcal{E}(z,s)\) is defined in Lemma \ref{Lemma4.2} and \(NL(\eta)(z,s):=NL(\varepsilon)(y,s)\) (\(NL(\varepsilon)(y,s)\) is defined in \ref{NL}). Before we perform the energy estimate, we prove a useful lemma.
\begin{lemma}
\label{lem:sobolev type}
    We have the estimate
    \begin{align*}
        \int_{|z| \geq K}\frac{|(z \cdot \nabla)^k \eta|^2}{|z|^{\gamma}}\frac{dz}{|z|^2}\leq \left(\frac{2}{\gamma}\right)^2 \int_{|z| \geq K}\frac{|(z \cdot \nabla)^{k+1} \eta|^2}{|z|^{\gamma}}\frac{dz}{|z|^2}+C_{k+1}e^{-\frac{5}{12}s}, \quad k=0,1,2, \quad \gamma=8.01.
    \end{align*}
\end{lemma}
\begin{proof}
    This follows directly from considering 
    \begin{align*}
        \int_{|z| \geq K}\frac{|(z \cdot \nabla)^{k+1} \eta-\xi(z \cdot \nabla)^{k} \eta|^2}{|z|^{\gamma}}\frac{dz}{|z|^2} \geq 0, \quad \xi \in \mathbb{R},
    \end{align*}
    and performing integration by parts, then optimizing the choice of $\xi$.
\end{proof}

\begin{lemma}
\label{lem:energy estim z 1}
    We have the energy estimates
    \begin{align}
    &\frac{1}{2}\frac{d}{ds}\|\eta\|_{\natural}^2 \leq -0.00125\|\eta\|_{\natural}^2+C_1e^{-\frac{5}{6}s}, \label{inq:int2 energy 1}\\
    &\frac{1}{2}\frac{d}{ds}\|(z \cdot \nabla)\eta\|_{\natural}^2 \leq -0.00125\|(z \cdot \nabla)\eta\|_{\natural}^2+C_2e^{-\frac{5}{6}s},\label{inq:int2 energy 2}\\
    &\frac{1}{2}\frac{d}{ds}\|(z \cdot \nabla)^2\eta\|_{\natural}^2 \leq -0.00125\|(z \cdot \nabla)^2\eta\|_{\natural}^2+C_3e^{-\frac{5}{6}s},\label{inq:int2 energy 3}\\
    &\frac{1}{2}\frac{d}{ds}\|(z \cdot \nabla)^3\eta\|_{\natural}^2 \leq -0.00125\|(z \cdot \nabla)^3\eta\|_{\natural}^2+C_3e^{-\frac{5}{6}s}.\label{inq:int2 energy 4}
    \end{align}
\end{lemma}
\begin{proof}
We only prove \ref{inq:int2 energy 1}, the other energy estimates can be proved similarly using the parabolic regularity result (See Theorem \ref{thm:reg main} and the following remarks). Note that we have
\begin{align*}
    \frac{1}{2}\frac{d}{ds}\int_{|z| \geq K}\frac{|\eta|^2}{|z|^{\gamma}}\frac{dz}{|z|^2}&=-\frac{1}{4}\int_{|z| \geq K}\frac{\eta \cdot (z \cdot \nabla)\eta}{|z|^{\gamma}}\frac{dz}{|z|^2}+\frac{1}{3}\int_{|z| \geq K}\frac{|\eta|^2}{|z|^{\gamma}}\frac{dz}{|z|^2}+e^{-s/2}\int_{|z| \geq K}\frac{\eta \Delta \eta}{|z|^{\gamma}}\frac{dz}{|z|^2}\\
    & \quad +\int_{|z| \geq K}\frac{\eta \cdot E}{|z|^{\gamma}}\frac{dz}{|z|^2}+\int_{|z| \geq K}\frac{\eta \cdot NL(\eta)}{|z|^{\gamma}}\frac{dz}{|z|^2}.
\end{align*}
By integration by parts, we have
\begin{align*}
    -\frac{1}{4}\int_{|z| \geq K}\frac{\eta \cdot (z \cdot \nabla)\eta}{|z|^{\gamma}}\frac{dz}{|z|^2}=\left(-\frac{\gamma}{8}\right)\int_{|z| \geq K}\frac{|\eta|^2}{|z|^2}\frac{dz}{|z|^2}+C_1e^{-\frac{5}{6}s}.
\end{align*}
For the Laplacian term, integrating by parts yields
\begin{align*}
    e^{-s/2}\int_{|z| \geq K}\frac{\eta \Delta \eta}{|z|^{\gamma}}\frac{dz}{|z|^2}=-e^{s/2}\int_{|z|\geq K}\frac{|\nabla \eta|^2}{|z|^{\gamma}}\,dz+o(e^{-\frac{5}{6}s}).
\end{align*}
Next, by using the estimates in Lemma \ref{Lemma4.1} (iii), (iv) and Lemma \ref{Lemma4.2}, we have
\begin{align*}
    &\int_{|z| \geq K}\frac{\eta \cdot \mathcal{E}}{|z|^{\gamma}}\frac{dz}{|z|^2}+\int_{|z| \geq K}\frac{\eta \cdot NL(\eta)}{|z|^{\gamma}}\frac{dz}{|z|^2}\\
    &=\int_{K \leq |z| \leq e^{\frac{5}{49}s}}\frac{\eta \cdot \mathcal{E}}{|z|^{\gamma}}\frac{dz}{|z|^2}+\int_{K \leq |z| \leq e^{\frac{5}{49}s}}\frac{\eta \cdot NL(\eta)}{|z|^{\gamma}}\frac{dz}{|z|^2}\\
    &\quad+ \int_{e^{\frac{5}{49}s} \leq |z| \leq e^{\frac{1}{4}s}}\frac{\eta \cdot \mathcal{E}}{|z|^{\gamma}}\frac{dz}{|z|^2} +\int_{e^{\frac{5}{49}s} \leq |z| \leq e^{\frac{1}{4}s}}\frac{\eta \cdot NL(\eta)}{|z|^{\gamma}}\frac{dz}{|z|^2} \leq\frac{2}{\mathcal{P}_{\theta}^3}\int_{|z| \geq K}\frac{|\eta|^2}{|z|^{\gamma}}\frac{dz}{|z|^2}+o(e^{-\frac{5}{6}s}).
\end{align*}
Here, the technical constant \(\frac{5}{49}\) is chosen to simplify the nonlinear estimate for \(K \leq |z| \leq e^{\frac{5}{49}s}\) and to control the nonlinear estimate for \(|z| \geq e^{\frac{5}{49}s}\). Combining the above estimates and using the fact that \(\gamma=8.01\) yields
\begin{align*}
        \frac{1}{2}\frac{d}{ds}\|\eta\|_{\natural}^2 \leq -0.00125\|\eta\|_{\natural}^2+C_1e^{-\frac{5}{6}s},
    \end{align*}
    which is the desired energy estimate.
\end{proof}
Finally, we can improve the damping in the above energy estimates to get the following Lemma.
\begin{lemma} \label{lem:intEngz2}
%\label{lem: energy estim z damp }
We have the following improved energy estimates
    \begin{align}
    &\frac{1}{2}\frac{d}{ds}\left(\|\eta\|_{\natural}^2+100\|(z \cdot \nabla)\eta\|_{\natural}^2\right) \leq -\frac{1}{2}\|\eta\|_{\natural}^2+\frac{C_1^2}{K}e^{-\frac{5}{6}s},\label{inq:energy z 1}\\
    &\frac{1}{2}\frac{d}{ds}\left(\|(z \cdot \nabla)\eta\|_{\natural}^2+100\|(z \cdot \nabla)^2\eta\|_{\natural}^2\right) \leq -\frac{1}{2}\|(z \cdot \nabla)\eta\|_{\natural}^2+\frac{C_2^2}{K}e^{-\frac{5}{6}s}, \label{inq:energy z 2}\\
    &\frac{1}{2}\frac{d}{ds}\left(\|(z \cdot \nabla)^2\eta\|_{\natural}^2+100\|(z \cdot \nabla)^3\eta\|_{\natural}^2\right) \leq -\frac{1}{2}\|(z \cdot \nabla)^2\eta\|_{\natural}^2+\frac{C_3^2}{K}e^{-\frac{5}{6}s}.\label{inq:energy z 3}
\end{align}
\end{lemma}
\begin{proof}
    Adding \ref{inq:int2 energy 1} and \ref{inq:int2 energy 2} multiplied by 100, combined with Lemma \ref{lem:sobolev type} yields \ref{inq:energy z 1}. We note that \ref{inq:energy z 2} and \ref{inq:energy z 3} can be proved similarly by adding the corresponding equations.
\end{proof}

\subsection{Control of \(\eta^{ex}\)}

\label{subsec:vepe}
This section is devoted to the control of \(\eta^{ex}(z,s)\) in the \(z\) variables, where \(z=ye^{-s/4}\). Using the comparison principle, we improve the constant \(D_1\) and consequently verify that the bootstrap condition (\ref{out2}) holds strictly for all \(s \in [s_0,s_1]\). Note that by definition, we have
\begin{align*}
    \eta^{ex}(z,s)=\eta(z,s)(1-\mchi_{K}(ze^{-\frac{5}{49}s})),
\end{align*}
and by a direct calculation, we have the PDE satisfied by \(\vepe\)
\begin{align}
\label{eq:pdevepe}
    \partial_s \eta^{ex}=-\frac{1}{4}(z \cdot \nabla)\eta^{ex}+\frac{1}{3} \eta^{ex}+e^{-s/2}\Delta \eta^{ex}+\mathcal{E}^{ex}+R^{ex}+NL^{ex}(\eta),
\end{align}
for \(|z| \geq K\), where \(\mathcal{E}^{ex}:=\mathcal{E}(1-\mchi_{K}(ze^{-\frac{5}{49}s}))\), \(NL^{ex}(\eta):=NL(\eta)(1-\mchi_{K}(ze^{-\frac{5}{49}s}))\) and 
\begin{align*}
    R^{ex}:=-\frac{1}{4}(z \cdot \nabla)\mchi_{K}(ze^{-\frac{5}{49}s}) \cdot \eta+e^{-s/2}\nabla \eta \cdot \nabla \mchi_K(ze^{-\frac{5}{49}s})+e^{-s/2} \eta \Delta \mchi_K(ze^{-\frac{5}{49}s}).
\end{align*}
Note that the generated error \(\mathcal{E}\) is defined in Lemma \ref{Lemma4.2}, and the nonlinear term \(NL(\eta)\) is defined in Lemma \ref{Lemma4.1} (iv). By the definition of \ref{eq:pdevepe}, we may write the PDE \ref{eq:pdevepe} as
\begin{align*}
    \partial_s \eta^{ex}=-\frac{1}{4}(z \cdot \nabla)\eta^{ex}+\frac{1}{3}\eta^{ex}+e^{-s/2}\Delta \eta^{ex}+f(z,s),
\end{align*}
for \(|z| \geq Ke^{\frac{5}{49}s}\), where \(f(z,s)=\mathcal{E}^{ex}+R^{ex}+NL^{ex}(\eta)\).

\begin{lemma}
\label{lem:estmf}
    \(f\) satisfies the estimate
    \begin{align*}
    &|f(z,s)| \lesssim e^{-\frac{47}{5880}s}, \quad |z| \geq Ke^{\frac{5}{49}s},
    \end{align*}
    for all \(s \in [s_0,s_1]\), where \(f(z,s)=\mathcal{E}^{ex}+R^{ex}+NL^{ex}(\eta)\).
\end{lemma}
\begin{proof}
    By Lemma \ref{Lemma4.2}, we have \(|\mathcal{E}^{ex}| \lesssim \min{\{e^{-s/3},e^{-s/2}|z|^2\}} \) for \(|z| \geq Ke^{\frac{5}{49}s}\). Next, by Remark \ref{rmk:parabolic z variables} and the estimate in Lemma \ref{Lemma4.1} (ii), we have \(|R^{ex}| \lesssim e^{-\frac{47}{5880}s}\) for \(Ke^{\frac{5}{49}s} \leq|z| \leq 2Ke^{\frac{5}{49}s}\). As for the nonlinear term, by Lemma \ref{Lemma4.1} (iv), we have \(|NL^{ex}(\eta)| \lesssim \frac{1}{\mathcal{P}_{\theta}^2} \lesssim e^{-\frac{11}{147}s}\) for \(|z| \geq Ke^{\frac{5}{49}s}\). Combining the above yields the desired bound.
\end{proof}

Next, we introduce the operator \(S\) acting on smooth functions by
\begin{align*}
    S \phi=-\partial_s \phi-\frac{1}{4}(z \cdot \nabla)\phi+\frac{1}{3}\phi+e^{-s/2}\Delta\phi+f(z,s).
\end{align*}
Note that in this setting, we have
\begin{align*}
    S(\eta^{ex})=0.
\end{align*}

\begin{lemma}
\label{lem:compari}
    Define the comparison function \(q(z,s)\) as
\begin{align*}
    q(z,s)=\frac{1}{2}D_1\epsilon e^{-\frac{1}{6}s}(1-e^{-\epsilon s})\left(1-e^{-(1-|ze^{-s/4}|)^2}\right)^2|z|^2,
\end{align*}
where \(\epsilon>0\) is a small fixed constant. Then we have \(S(q) \leq 0\) for \(|z| \geq Ke^{\frac{5}{49}s}\).
    
\end{lemma}
\begin{proof}
    By direct calculation, we have
    \begin{align*}
        S(q)&=-\frac{1}{2}D_1\epsilon^2e^{-\frac{1}{6}s}(1-e^{-\epsilon s})\left(1-e^{-(1-|ze^{-s/4}|)^2}\right)^2|z|^2+(e^{-s/2}\Delta q+f).
    \end{align*}
    Note that since \(e^{-s/2}\Delta q+f=0\) on \(|z|=e^{s/4}\) and satisfies \(|e^{-s/2}\Delta q+f| \lesssim e^{-\frac{47}{5880}s}\) for \(|z| \geq Ke^{\frac{5}{49}s}\), it follows that \(S(q) \leq 0\) for \(|z| \geq Ke^{\frac{5}{49}s}\).
\end{proof}
\begin{remark}
\label{rmk:S(-q)geq0}
    We also have \(S(-q) \geq 0\) for \(|z| \geq Ke^{\frac{5}{49}s}\).
\end{remark}

\begin{lemma}
\label{Lemma5.10}
Let \(\varepsilon(s) \in S(s)\), \(s \in [s_0.s_1]\). Then it holds that
    \begin{align}
        |\eta^{ex}(z,s)| \leq  \frac{1}{2}\epsilon D_1 e^{-\frac{1}{6}s}(1-e^{-\epsilon s})\left(1-e^{-(1-|ze^{-s/4}|)^2}\right)^2|z|^2, \label{max1}
    \end{align}
    for \(|z| \geq Ke^{\frac{5}{49}s}\).
\end{lemma}
\begin{proof}
    Define \(\Gamma\) as
\begin{align*}
        \Gamma=\left(\bigcup_{s_0 \leq s \leq s_1} R_s\right) \cup \: \left(\{|z|<e^{\frac{1}{4}s_0}\} \times \{s_0\}\right),
\end{align*}
where \(R_s=\{z\in\mathbb{R}^2:|z|=e^{\frac{s}{4}}\}\). Note that we have
    \begin{align}
         -q(z,s)\leq\eta^{ex}(z,s) \equiv 0 \leq q(z,s), \quad \text{on } \Gamma.
    \end{align}
    Applying the comparison principle together with the estimate in Lemma \ref{lem:compari} yields:
    \begin{align*}
        \eta^{ex}(z,s) \leq \frac{1}{2}D_1\epsilon e^{-\frac{1}{6}s}(1-e^{-\epsilon s})\left(1-e^{-(1-|ze^{-s/4}|)^2}\right)^2|z|^2, \quad |z| \geq Ke^{\frac{5}{49}s}.
    \end{align*}
Similarly, applying the comparison principle together with the estimate in Remark \ref{rmk:S(-q)geq0} yields
\begin{align*}
   -\frac{1}{2}D_1\epsilon e^{-\frac{1}{6}s}(1-e^{-\epsilon s})\left(1-e^{-(1-|ze^{-s/4}|)^2}\right)^2|z|^2 \leq \eta^{ex}(z,s), \quad |z| \geq Ke^{\frac{5}{49}s}.
\end{align*}
Thus, we have
\begin{align}
        |\eta^{ex}(z,s)| \leq \frac{1}{2}D_1\epsilon e^{-\frac{1}{6}s}(1-e^{-\epsilon s})\left(1-e^{-(1-|ze^{-s/4}|)^2}\right)^2|z|^2, \quad |z| \geq Ke^{\frac{5}{49}s},
\end{align}
which proves Lemma \ref{Lemma5.10}.
\end{proof}
\begin{remark}
    Thanks to the factor \((1-e^{-(1-|ze^{-s/4}|)^2})^2\), we have \(\eta^{ex} \equiv 0\) at \(|z|=e^{\frac{s}{4}}\), which matches the boundary condition (see \ref{eq:wysboundary}).
\end{remark}

\paragraph{\textbf{Proof of Proposition \ref{Proposition3.2}:}} It follows from Lemma \ref{Lemma5.10}, combined with integrating
\ref{en1}, \ref{en2}, \ref{en6}, \ref{en7} \ref{en3}, \ref{en4}, \ref{en5}, \ref{inq:energy z 1}, \ref{inq:energy z 2} and \ref{inq:energy z 3} forward in time. \qed

\vspace{10pt}

\paragraph{\textbf{Acknowledgment:}} The research of V. T. Nguyen is supported by the National Science and Technology Council of Taiwan: NSTC 114-2628-M-002-002.

%The authors also acknowledge the use of ChatGPT (OpenAI) for exploratory discussion regarding the comparison function in Lemma \ref{lem:compari}. All strategies, mathematical arguments were independently developed and rigorously verified by the authors.

\vspace{5pt}

\paragraph{\textbf{Disclosure of interest:}} The authors report there are no competing interests to declare. 

\bibliographystyle{unsrt}
\bibliography{ref}

@article{kavallaris2018non,
  title={Non-local partial differential equations for engineering and biology},
  author={Kavallaris, Nikos I and Suzuki, Takashi},
  journal={Mathematical Modeling and Analysis},
  volume={31},
  year={2018},
  publisher={Springer}
}

@article{nguyen2025construction,
  title={Construction of type {I}-Log blowup for the Keller-Segel system in dimensions 3 and 4},
  author={Nguyen, Van Tien and Nouaili, Nejla and Zaag, Hatem},
  journal={Annals of PDE},
  volume={11},
  number={1},
  pages={12},
  year={2025},
  publisher={Springer}
}

@article{merle2024degenerate,
  title={On degenerate blow-up profiles for the subcritical semilinear heat equation},
  author={Merle, Frank and Zaag, Hatem},
  journal={Journal of the European Mathematical Society},
  year={2024}
}

@article{guo2009nonlocal,
  title={A nonlocal quenching problem arising in a micro-electro mechanical system},
  author={Guo, Jong-Shenq and Hu, Bei and Wang, Chi-Jen},
  journal={Quarterly of applied mathematics},
  volume={67},
  number={4},
  pages={725--734},
  year={2009}
}

@article{duong2019profile,
  title={Profile of a touch-down solution to a nonlocal MEMS model},
  author={Duong, Giao Ky and Zaag, Hatem},
  journal={Mathematical Models and Methods in Applied Sciences},
  volume={29},
  number={07},
  pages={1279--1348},
  year={2019},
  publisher={World Scientific}
}

@article{filippas1993quenching,
  title={Quenching profiles for one-dimensional semilinear heat equations},
  author={Filippas, Stathis and Guo, Jong-Shenq},
  journal={Quarterly of applied mathematics},
  volume={51},
  number={4},
  pages={713--729},
  year={1993}
}

@article{guo2014recent,
  title={Recent developments on a nonlocal problem arising in the micro-electro mechanical system},
  author={Guo, Jong-Shenq},
  journal={Tamkang Journal of Mathematics},
  volume={45},
  number={3},
  pages={229--241},
  year={2014}
}

@article{merle1997reconnection,
  title={Reconnection of vortex with the boundary and finite time quenching},
  author={Merle, Frank and Zaag, Hatem},
  journal={Nonlinearity},
  volume={10},
  number={6},
  pages={1497--1550},
  year={1997},
  publisher={Bristol [England]: Institute of Physics and the London Mathematical Society~…}
}

@article{guo2005touchdown,
  title={Touchdown and pull-in voltage behavior of a MEMS device with varying dielectric properties},
  author={Guo, Yujin and Pan, Zhenguo and Ward, Michael Jeffrey},
  journal={SIAM Journal on Applied Mathematics},
  volume={66},
  number={1},
  pages={309--338},
  year={2005},
  publisher={SIAM}
}

@article{guo2012nonlocal,
  title={On a nonlocal parabolic problem arising in electrostatic MEMS control},
  author={Guo, Jong-Shenq and Kavallaris, Nikos I},
  journal={Discrete Contin. Dyn. Syst},
  volume={32},
  number={5},
  pages={1723--1746},
  year={2012}
}

@article{merle1992solution,
  title={Solution of a nonlinear heat equation with arbitrarily given blow-up points},
  author={Merle, Frank},
  journal={Communications on pure and applied mathematics},
  volume={45},
  number={3},
  pages={263--300},
  year={1992},
  publisher={Wiley}
}

@article{duong2018construction,
  title={Construction of a stable blowup solution with a prescribed behavior for a non-scaling-invariant semilinear heat equation},
  author={Duong, Giao Ky and Nguyen, Van Tien and Zaag, Hatem},
  journal={Tunisian Journal of Mathematics},
  volume={1},
  number={1},
  pages={13--45},
  year={2018},
  publisher={Mathematical Sciences Publishers}
}

@article{filippas1992refined,
  title={Refined asymptotics for the blowup of $u_t—\Delta u= u^p$},
  author={Filippas, Stathis and Kohn, Robert V},
  journal={Communications on pure and applied mathematics},
  volume={45},
  number={7},
  pages={821--869},
  year={1992},
  publisher={Wiley Online Library}
}

@book{esposito2010mathematical,
  title={Mathematical analysis of partial differential equations modeling electrostatic MEMS},
  author={Esposito, Pierpaolo and Ghoussoub, Nassif and Guo, Yujin},
  volume={20},
  year={2010},
  publisher={American Mathematical Soc.}
}

@article{velazquez1993classification,
  title={Classification of singularities for blowing up solutions in higher dimensions},
  author={Vel{\'a}zquez, JJL},
  journal={Transactions of the American Mathematical Society},
  volume={338},
  number={1},
  pages={441--464},
  year={1993}
}

@article{flores2007analysis,
  title={Analysis of the dynamics and touchdown in a model of electrostatic MEMS},
  author={Flores, Gilberto and Mercado, G and Pelesko, John A and Smyth, N},
  journal={SIAM Journal on Applied Mathematics},
  volume={67},
  number={2},
  pages={434--446},
  year={2007},
  publisher={SIAM}
}

@book{pelesko2002modeling,
  title={Modeling {MEMS} and {NEMS}},
  author={Pelesko, John A and Bernstein, David H},
  year={2002},
  publisher={CRC press}
}

@article{deng1989blow,
  title={On the blow up of u\_ $\{$t$\}$ at quenching},
  author={Deng, Keng and Levine, Howard A},
  journal={Proceedings of the American Mathematical Society},
  volume={106},
  number={4},
  pages={1049--1056},
  year={1989}
}

@article{GUO1992507,
title = {The critical length for a quenching problem},
journal = {Nonlinear Analysis: Theory, Methods \& Applications},
volume = {18},
number = {6},
pages = {507-516},
year = {1992},
issn = {0362-546X},
doi = {https://doi.org/10.1016/0362-546X(92)90207-U},
url = {https://www.sciencedirect.com/science/article/pii/0362546X9290207U},
author = {Jong-Shenq Guo}
}

@article{collot2017strongly,
  title={On strongly anisotropic type {II} blow up},
  author={Collot, Charles and Merle, Frank and Rapha{\"e}l, Pierre},
  journal={J. Amer. Math. Soc.},
  year={2020},
volume = {33}, 
pages = {527--607}
}

@article{del2018sign,
  title={Sign-changing blowing-up solutions for the critical nonlinear heat equation},
  author={Del Pino, Manuel and Musso, Monica and Wei, Juncheng and Zheng, Youquan},
  journal={Annali della Scuola Normale Superiore di Pisa},
  volume = {XXI},
  pages={569-641},
  year={2020}
}

@article{del2011large,
  title={Large energy entire solutions for the Yamabe equation},
  author={Del Pino, Manuel and Musso, Monica and Pacard, Frank and Pistoia, Angela},
  journal={Journal of Differential Equations},
  volume={251},
  number={9},
  pages={2568--2597},
  year={2011},
  publisher={Elsevier}
}

@article{del2013torus,
  title={Torus action on {$S^n$} and sign-changing solutions for conformally invariant equations},
  author={Del Pino, Manuel and Musso, Monica and Pacard, Frank and Pistoia, Angela},
  journal={Annali della Scuola Normale Superiore di Pisa-Classe di Scienze},
  volume={12},
  number={1},
  pages={209--237},
  year={2013}
}

@inproceedings{herrero1993blow,
  title={Blow-up behaviour of one-dimensional semilinear parabolic equations},
  author={Herrero, Miguel A and Vel{\'a}zquez, Juan JL},
  booktitle={Annales de l'Institut Henri Poincar{\'e} C, Analyse non lin{\'e}aire},
  volume={10},
  number={2},
  pages={131--189},
  year={1993},
  organization={Elsevier}
}

@inproceedings{filippas1993blowup,
  title={On the blowup of multidimensional semilinear heat equations},
  author={Filippas, Stathis and Liu, Wenxiong},
  booktitle={Annales de l'Institut Henri Poincar{\'e} C, Analyse non lin{\'e}aire},
  volume={10},
  number={3},
  pages={313--344},
  year={1993},
  organization={Elsevier}
}

@book{krylov1996lectures,
  title={Lectures on elliptic and parabolic equations in Holder spaces},
  author={Nikola Krylov},
  volume={12},
  year={1996},
  publisher={American Mathematical Soc.}
}

@article{krylov1981certain,
  title={A certain property of solutions of parabolic equations with measurable coefficients},
  author={Krylov, Nicolai V and Safonov, Mikhail V},
  journal={Mathematics of the USSR-Izvestiya},
  volume={16},
  number={1},
  pages={151},
  year={1981},
  publisher={IOP Publishing}
}

\appendix
\section{Parabolic regularity}
\label{appdx}
In this section, we establish several parabolic regularity results that will be used throughout the paper. We begin by introducing the notations, followed by stating the parabolic Schauder estimates and the Krylov–Safonov estimates. Let \(Q_r\) be a parabolic cylinder of the form
\begin{align*}
    Q_r=B(0,r) \times \{-r^2 \leq s \leq 0\} \subset \mathbb{R}^n \times (-\infty,0],
\end{align*}
where \(B(0,r)\) is the open ball centered at 0 with radius \(r>0\) in \(\mathbb{R}^n\). For a point \(p=(z,s) \in Q_r\), we define the parabolic metric \(|p|\) as 
\begin{align*}
    |p|=|z|+|s|^{\frac{1}{2}}.
\end{align*}
Let \(R:Q_r \to \mathbb{R}\) be a function, we define the H\"older semi-norm of exponent \(\delta \in (0,1)\) as
\begin{align*}
    [R]_{C^{0,\delta}(Q_r)}=\sup_{p,q \in Q_r,\;p \neq q} \frac{|R(p)-R(q)|}{|p-q|^{\delta}},
\end{align*}
and define the H\"older norm of exponent \(\delta \in (0,1)\) as
\begin{align*}
    \|R\|_{C^{0,\delta}(Q_r)} = \|R\|_{L^{\infty}(Q_r)}+[R]_{C^{0,\delta}(Q_r)}. 
\end{align*}
We now state the well-known parabolic Schauder estimates (see \cite{krylov1996lectures} for the Schauder theory).
\begin{proposition}[Schauder]
\label{prop:Schauder}
    Suppose \(f\in C^3(Q_r)\) satisfies a parabolic equation
    \begin{align*}
        \partial_sf=Tf+R,
    \end{align*}
    where \(R \in C^1(Q_r)\) and \(T\) is a second order linear operator of the form
    \begin{align*}
       T=a^{ij}(z,s)\partial^2_{ij}+b^i(z,s)\partial_i+c(z,s),
    \end{align*}
    with \(a^{ij}(z,s)=a^{ji}(z,s)\). Moreover, suppose there is \(\Lambda>0\) such that
    \begin{align*}
        \|a^{ij}\|_{C^{0,\delta}(Q_r)} \leq \Lambda, \quad \|b^i\|_{C^{0,\delta}(Q_r)} \leq \Lambda, \quad \|c\|_{C^{0,\delta}(Q_r)} \leq \Lambda,
    \end{align*}
    and 
    \begin{align*}
        \Lambda^{-1}|\xi|^2 \leq a^{ij}(z,s)\xi_i \xi_j \leq \Lambda|\xi|^2,
    \end{align*}
    for all \(\xi \in \mathbb{R}^n\). Then for every for \(\delta \in (0,1)\), we have the following estimate
\begin{align*}
    \sup_{Q_{r/2}}|\partial_if|+\sup_{Q_{r/2}}|\partial^2_{ij} f| \leq C'\left(\|f\|_{L^{\infty}(Q_{r})}+\|R\|_{C^{0,\delta}(Q_{r})}\right),
\end{align*}
where \(C'=C'(\delta, \Lambda,n,r)>0\) is a constant depending only on \(\delta, \Lambda, n,r\).
\end{proposition}

We note that if the coefficients of the second-order linear operator \(T\) are merely measurable and bounded, the Schauder estimates do not hold in general; instead, we have the Krylov-Safonov Estimate.(see \cite{krylov1981certain} for the original paper) .

\begin{proposition}[Krylov-Safonov]
\label{prop:Krylov-Safonov}
    Let \(f \in C^3(Q_r)\) be a solution to the following parabolic PDE
    \begin{align*}
        \partial_sf=Qf+R,
    \end{align*}
    where \(Q\) is a second-order linear operator of the form 
    \begin{align*}
       Q=a^{ij}(z,s)\partial^2_{ij}+b^i(z,s)\partial_i+c(z,s),
    \end{align*}
    with \(a^{ij}(z,s)=a^{ji}(z,s)\) and \(R \in L^3(Q_r)\). Moreover, suppose there is \(\Lambda>0\) such that
    \begin{align*}
        \|a^{ij}\|_{L^{\infty}(Q_{r})} \leq \Lambda, \quad \|b^i\|_{L^{\infty}(Q_r)} \leq \Lambda, \quad \|c\|_{L^{\infty}(Q_r)} \leq \Lambda,
    \end{align*}
    and 
    \begin{align*}
        \Lambda^{-1}|\xi|^2 \leq a^{ij}(z,s)\xi_i \xi_j \leq \Lambda|\xi|^2,
    \end{align*}
    for all \(\xi \in \mathbb{R}^n\), then there is \(\delta \in (0,1)\), depending on \(n\) and \(\Lambda\) such that
    \begin{align*}
        [f]_{C^{0,\delta}(Q_{r/2})} \leq C''\left(\|f\|_{L^{\infty}(Q_{r})}+\|R\|_{L^3(Q_{r})}\right),
    \end{align*}
    where \(C''=C''(n,\Lambda,\delta,r)>0\).
\end{proposition}

To better state our main result regarding parabolic regularity, we recall some notations and estimates. Note that by writing \ref{selfsimpde} in the \(z\) variables, we have, for \(|z| \lesssim 1\),
    \begin{align}
    \label{eq:pde eta}
        \partial_s \eta= L \eta+\mathcal{E},
    \end{align}
    where the second-order operator \(L\) is given by 
    \begin{align}
    \label{def:mathcal P}
        L=e^{-s/2}\Delta-\frac{1}{4}(z \cdot \nabla)+P(z,s).
    \end{align}
    Here, \(P(z,s)\) is defined as 
    \begin{align*}
        P(z,s)=V(z,s)+\sum_{k=2}^{\infty}\binom{-2}{k}\frac{\eta^{k-1}}{\mathcal{P}_{\theta}^{k+2}},
    \end{align*}
    where \(V(z,s)=\frac{1}{3}+\frac{2}{\mathcal{P}_{\theta}^3}\), \(\mathcal{P}_{\theta}\) is defined in \ref{qp} and \(\mathcal{E}(z,s)=E(y,s)\) (\(E(y,s)\) is given in \ref{E}) . Moreover, note that by the pointwise estimates in Lemma \ref{Lemma4.1} (iii), and the pointwise estimate of the generated error \(\mathcal{E}\) in Lemma \ref{Lemma4.2}, we have
    \begin{align}
        &|\eta(z,s)| \lesssim e^{-\frac{5}{12}s}|z|^9, \label{inq app :pointwise eta}\\
        &|\mathcal{E}(z,s)| \lesssim se^{-\frac{1}{2}s}\lesssim e^{-\frac{5}{12}s}|z|^9, \label{inq app :pointwise E}
    \end{align}
    in \(\{2Ke^{-\frac{1}{4}s} \leq |z| \leq 32K\} \times \{\frac{1}{2}s_0 \leq s \leq s_1\}:=A\). The following is our main result.
    \begin{theorem}
    \label{thm:reg main}
        We have the estimates 
        \begin{align*}
            |\partial_i\,\eta(z,s)| \lesssim e^{-\frac{5}{12}s}|z|^8, \quad |\partial^2_{ij}\,\eta(z,s)| \lesssim e^{-\frac{5}{12}s}|z|^7, \quad i,j=1,2
        \end{align*}
        for \((z,s) \in \{4Ke^{-s/4} \leq |z| \leq 16K\} \times \{s_0 \leq s \leq s_1\}:=A'\).
    \end{theorem}
    
    \begin{proof}
        Let \((\bar{z},\bar{s}) \in A'\) and let \(\lambda=|\bar{z}|\). We define the rescaled functions \(\widetilde{\varepsilon}\), \(\widetilde{\mathcal{E}}\) and \(\widetilde{P}\) as
        \begin{align*}
            \widetilde{\eta}(z,s)=\eta(\bar{z}+\lambda z,\bar{s}+\lambda^2s),\quad \widetilde{\mathcal{E}}(z,s)=\mathcal{E}(\bar{z}+\lambda z,\bar{s}+\lambda^2s), \quad \widetilde{P}(z,s)=P(\bar{z}+\lambda z,\bar{s}+\lambda^2s),
        \end{align*}
        then by \ref{inq app :pointwise eta} and \ref{inq app :pointwise E}, we have
        \begin{align*}
            |\widetilde{\eta}(z,s)| \lesssim e^{-\frac{5}{12}(\bar{s}+\lambda^2s)}|\bar{z}+\lambda z|^9, \quad |\widetilde{\mathcal{E}}(z,s)| \lesssim e^{-\frac{5}{12}(\bar{s}+\lambda^2s)}|\bar{z}+\lambda z|^9,
        \end{align*}
        for \((z,s) \in Q_1\). Moreover, \(\widetilde{\eta}\) satisfies the PDE 
        \begin{align*}
            \partial_s \widetilde{\eta}=\widetilde{L}\,\widetilde{\eta}+\lambda^2\widetilde{\mathcal{E}}(z,s),
        \end{align*}
        where the operator \(\widetilde{L}\) is is given by
        \begin{align*}
            \widetilde{L}=e^{-\frac{1}{2}(\bar{s}+\lambda^2s)} \Delta-\frac{1}{4}\lambda(\bar{z}+\lambda z) \cdot \nabla+\lambda^2\widetilde{P}(z,s).
        \end{align*}
        Note that there is \(\delta \in (0,1)\) such that the 
        \(C^{0,\delta}(Q_{1/2})\) norm of the coefficients of the operator \( \widetilde{L}\) are uniformly bounded by for all \((\bar{z},\bar{s})\in A'\) (See Lemma \ref{lem:C alpha regularity}). Hence, for all \((\bar{z},\bar{s}) \in A'\), by parabolic Schauder Estimates \ref{prop:Schauder}, there is a constant \(C'=C'(\delta)\) such that
        \begin{align*}
            \sup_{Q_{1/4}}|\partial_i \widetilde{\eta}|+\sup_{Q_{1/4}}|\partial^2_{ij}\widetilde{\eta}| &\leq C'\left(\|\widetilde{\eta}\|_{L^{\infty}(Q_{1/2})}+\|\widetilde{\mathcal{E}}\|_{L^{\infty}(Q_{1/2})}\right) \lesssim e^{-\frac{5}{12}\bar{s}}|\bar{z}|^9.
        \end{align*}
        Thus, we have 
        \begin{align*}
            |\partial_i\widetilde{\eta}(0,0)|+|\partial^2_{ij}\widetilde{\eta}(0,0)| \lesssim e^{-\frac{5}{12}\bar{s}}|\bar{z}|^9,
        \end{align*}
        for all \((\bar{z},\bar{s}) \in A'\), which implies
        \begin{align*}
            &|\partial_i\,\eta(z,s)| \lesssim e^{-\frac{5}{12}s}|z|^8, \quad |\partial^2_{ij}\,\eta(z,s)| \lesssim e^{-\frac{5}{12}s}|z|^7, \quad i,j=1,2
        \end{align*}
        for \((z,s) \in A'\).
    \end{proof}

    \begin{lemma}
     \label{lem:C alpha regularity}
         There is \(\delta \in (0,1)\), such that the 
        \(C^{0,\delta}(Q_{1/2})\) norm of the coefficients of the operator \(\widetilde{L}\) given in \ref{def:mathcal P} are uniformly bounded for all \((\bar{z},\bar{s}) \in A'\).
     \end{lemma}
     \begin{proof}
         It is sufficient to show for some \(\delta \in (0,1)\), it holds that
         \begin{align*}
             \|\lambda^2\widetilde{P}\|_{C^{0,\delta}(Q_{1/2})} < \infty, \quad \lambda=|\bar z|,
         \end{align*}
         uniformly for all \((\bar{z},\bar{s}) \in A'\), since the \(C^{0,\delta}(Q_{1/2})\) norm of other coefficients of \(\widetilde{L}\) clearly satisfies the bound. We start by establishing H\"older estimates for \(\widetilde{\eta}\). Note that \(\widetilde{\eta}\) satisfies the PDE 
         \begin{align*}
            \partial_s \widetilde{\eta}=\widetilde{L}\,\widetilde{\eta}+\lambda^2\widetilde{\mathcal{E}}(z,s),
        \end{align*}
        and by the estimate \ref{inq app :pointwise eta},
        the definition of the quenching profile \ref{qp}, we conclude that the coefficients of the operator \(\widetilde{L}\) are uniformly bounded for all \((\bar{z},\bar{s}) \in A'\). Thus, by the Krylov-Safonov estimates \ref{prop:Krylov-Safonov}, combined with the estimates \ref{inq app :pointwise eta} and \ref{inq app :pointwise E}, there is \(\delta>0\), and a constant \(C''=C''(\delta)>0\) such that 
        \begin{align*}
            [\widetilde{\eta}]_{C^{0,\delta}(Q_{1/2})} &\leq C''\left(\|\widetilde{\eta}\|_{L^{\infty}(Q_{1})}+\|\widetilde{\mathcal{E}}\|_{L^3(Q_{1})}\right)< 1,
        \end{align*}
        for all \((\bar{z},\bar{s}) \in A'\). Thus, using the above estimate and the fact that
        \(\epsilon>0\) in \ref{out2} is chosen small (so that \(\|\eta(z,s)/\mathcal{P}(z,s)\|_{L^{\infty}(\mathbb{R}^2)} \ll 1\)), we have
        \begin{align*}
            \|\lambda^2\widetilde{P}\|_{C^{0,\delta}(Q_{1/2})} &\leq 256K^2\|\widetilde{P}\|_{C^{0,\delta}(Q_{1/2})}\\
            & \leq 256K^2\left(\|\widetilde{V}\|_{C^{0,\delta}(Q_{1/2})}+\sum_{k=2}^{\infty}(k+2)\left\|\frac{\widetilde{\eta}^{k-1}}{\widetilde{\mathcal{P}_{\theta}}^{k+2}}\right\|_{C^{0,\delta}(Q_{1/2})}\right)< \infty,
        \end{align*}
        uniformly for all \((\bar{z},\bar{s}) \in A'\). 
        This concludes the proof of Lemma \ref{lem:C alpha regularity}.
     \end{proof}

    \begin{remark}
     \label{rmk:control of R}
         By deriving the PDE for \(\partial_i \eta(z,s)\), we can show
         \begin{align*}
             |\partial^3_{ijk}\eta(z,s)| \lesssim e^{-\frac{5}{12}s}|z|^6, \quad i,j,k=1,2,
         \end{align*}
         in \(\{8Ke^{-s/4} \leq |z| \leq 8K\}\).
     \end{remark}
     
     \begin{remark}
     \label{rmk: L2 rho control outer}
         Directly following by Theorem \ref{thm:reg main} and Remark \ref{rmk:control of R}, we have, in the \(y\) variables
         \begin{align*}
             &|\partial_i\varepsilon(y,s)| \lesssim e^{-\frac{32}{12}s}|y|^8,\quad |\partial^2_{ij}\varepsilon(y,s)| \lesssim e^{-\frac{32}{12}s}|y|^7, \quad i,j=1,2,
         \end{align*}
         in \(\{4K \leq |y| \leq 16Ke^{s/4}\}\), and the estimate
         \begin{align*}
             |\partial^3_{ijk}\varepsilon(y,s)| \lesssim e^{-\frac{32}{12}s}|y|^6, \quad i,j,k=1,2,
         \end{align*}
         in \(\{8K \leq |y| \leq 8Ke^{s/4}\}\).
     \end{remark}

     \begin{remark}
     \label{rmk: L2 rho control inner intermd control on boundary}
         By the bound for \(\varepsilon(y,s)\) in Lemma \ref{Lemma4.1} (i), and applying Schauder Estimates to equation \ref{lineq} , we have 
      \begin{align*}
         &\|\partial_i\varepsilon(s)\|_{L^{\infty}(|y| \leq 50K)} \lesssim C(K) e^{-\frac{32}{12}s}, \quad \|\partial^2_{ij}\varepsilon(s)\|_{L^{\infty}(|y| \leq 50K)} \lesssim C(K)e^{-\frac{32}{12}s},
      \end{align*}
      moreover, by deriving the PDE for \(\partial_i \varepsilon(y,s)\), we can also derive
      \begin{align*}
          \|\partial^3_{ijk}\varepsilon(s)\|_{L^{\infty}(|y| \leq 25K)} \lesssim C(K) e^{-\frac{32}{12}s}, \quad i,j,k=1,2.
      \end{align*}
     \end{remark}
     \begin{remark}
     \label{rmk:parabolic z variables}
      Let \(z'=ze^{-s/4}\) and by considering the PDE \ref{eq:pde eta} in the \(z'\) variables, we can conclude the estimates
      \begin{align*}
          |\partial_i \eta(z,s)| \lesssim e^{-\frac{5}{12}s}|z|^{3.005}, \quad |\partial_{ij}\eta(z,s)| \lesssim e^{-\frac{5}{12}s}|z|^{2.005}, \quad i,j=1,2,
      \end{align*}
      for \(4K \leq |z| \leq e^{\frac{s}{4}}\).
      Moreover, by deriving the PDE of \(\partial_i \varepsilon\) and \(\partial_{ij}\eta\), we also have
      \begin{align*}
          |\partial_{ijk}\,\eta(z,s)| \lesssim e^{-\frac{5}{12}s}|z|^{1.005}, \quad |\partial_{ijkl}\,\eta(z,s)| \lesssim e^{-\frac{5}{12}s}|z|^{0.005}, \quad i,j,k,l=1,2,
    \end{align*}
    for \( 8K \leq |z| \leq e^{\frac{s}{4}}\). Note that near \(|z|=e^{s/4}\), we used the parabolic Schauder estimates up to the boundary (In the original variables \(x\)) to conclude the estimate.
    \end{remark}
     
\medskip

\end{document}